\documentclass[12pt]{amsart}
\usepackage{amsmath,amssymb,bm,enumerate,graphicx,multido,pst-3dplot,pstricks,pst-node,caption,subcaption}
\usepackage[mathscr]{euscript}

\newtheorem{theorem}{Theorem}[section]

\newtheorem{lemma}[theorem]{Lemma}
\newtheorem{proposition}[theorem]{Proposition}
\newtheorem{conjecture}[theorem]{Conjecture}

\theoremstyle{definition}
\newtheorem{example}[theorem]{Example}
\newtheorem{construction}[theorem]{Construction}
\newtheorem{remark}[theorem]{Remark}

\def\D{{\mathscr{D}}}
\def\PP{{\mathscr{P}}}

\def\B{{\mathscr{B}}}
\def\C{{\mathscr{C}}}
\def\I{\mathscr{I}}
\def\H{\mathscr{H}}
\def\Aut{{\mathrm{Aut}}}
\def\Sym{{\mathrm{Sym}}}

\def\AGL{{\mathrm{AGL}}}

\def\CC{\bm{\mathscr{C}}}
\def\Bdy{\partial}
\def\c{{\rm c}}

\def\Bdy{\partial}

\usepackage{anysize}
\marginsize{2cm}{2cm}{2cm}{2cm}

\title{Recursive constructions for block-transitive, poset-imprimitive two-designs}
\author{Carmen Amarra$^{\rm\lowercase{b}}$}
\author{Alice Devillers$^{\rm\lowercase{a}}$ }
\author{Cheryl E. Praeger$^{\rm \lowercase{a},\ast}$}
\address{$^{\rm\lowercase{a}}$Centre for the Mathematics of Symmetry and Computation, The University of Western Australia, 35 Stirling Highway, Crawley, Western Australia 6009, Australia}
\email{alice.devillers@uwa.edu.au, cheryl.praeger@uwa.edu.au}
\address{$^{\rm\lowercase{b}}$ Institute of Mathematics, University of the Philippines Diliman, C.P. Garcia Avenue, Quezon City 1101, Philippines}
\email{mcamarra@math.upd.edu.ph}
\address{$^{\ast}$Corresponding author}

\date{\today}

\begin{document}

% These observations are summarised below in Theorem \ref{thm:wr}.

% \begin{theorem} \label{thm:wr}
% With the notation above, suppose that the poset $(I,\preccurlyeq)$ is the disjoint union $I = \widetilde{I} \,\dot\cup\, H$ of two non-empty sets such that for all $i\in H$ and $j \in \widetilde{I}$, we have $i>j$. Then
% \begin{enumerate}[(a)]
%     \item $H$ is an ancestral subset of $I$, and $F$ leaves invariant the partition $\C_{H}$ of $\PP$ with classes as in \eqref{class} (with $J = H$), and $F^{\C_{H}}$ is permutationally isomorphic to the permutation group $\ff[2]$ on $\Delta_{H}$.
%     \item For $j \in \widetilde{I}$, the ancestral set $A(j) = A'(j) \cup H$, where $A'(j) = A(j) \cap \widetilde{I}$ is the ancestral set consisting of the points above $j$ in the sub-poset $(\widetilde{I},\preccurlyeq)$, and the  group induced on each class $C \in \C_{H}$ by its setwise stabiliser $F_C$ is permutationally isomorphic to $F_{\widetilde{I}}$ in its action on $\Delta_{\widetilde{I}}$.
%     \item Moreover $F \cong F_{\widetilde{I}} \wr \ff[2]$.
% \end{enumerate}
% \end{theorem}

\begin{abstract}

We give two general constructions for  $2$-designs, that can be used recursively, and interchangeably, to produce new infinite families of $2$-designs admitting block-transitive groups of automorphisms which preserve arbitrarily large posets of partitions of the point-set. The only arbitrarily large posets for which constructions were previously known are chains of arbitrary length. Using the constructions  we exhibit new infinite families of poset-imprimitive block-transitive $2$-designs corresponding to several different arbitrarily large posets, as well as constructions for most posets with four nodes.
%     The input design $\widetilde{\D}$ admits imprimitive point-partitions which form a partially ordered set $\widetilde{\I}$ under the relation of refinement, while those for the output design $\D$ form a partially ordered set $\I$ that can be obtained from $\widetilde{\I}$ as follows: in the first construction, by adding a single node $s$ to $\widetilde{\I}$ satisfying the condition that $i \prec s$ for all $i \in \widetilde{\I}$; and in the second construction, by adding two nodes $1$ and $2$ to $\widetilde{\I}$ satisfying the condition that $1 \not\preccurlyeq 2$, $2 \not\preccurlyeq 1$, and $i \prec 1$  and $i \prec 2$ for all $i \in \widetilde{\I}$. For each construction we give necessary and sufficient conditions in order for the output $\D$ to be a $2$-design. We describe new families of $2$-designs, together with their corresponding posets $\I$, that can be obtained by recursively applying one or both constructions.
% {\color{red} Alice: it seems strange that in the first construction the node is called $s$ but in the second they are called $1,2$. That combined with the fact that in examples they are usually not called 1,2 makes me think it is a bad idea to call them 1,2. } {\color{blue}--In an earlier version the two nodes were called $r$ and $s$ - would that work? $r$ is the replication number, I'd prefer not to use that. I used $s_1,s_2$ now.}
    \bigskip %\noindent
    \begin{center}
    \emph{Dedicated to Alan Camina on the occasion of his eightieth birthday.}
    \end{center}

    \bigskip\noindent
    \emph{Keywords:}\quad $2$-design; block-transitive; poset-imprimitive groups. 
    
    \bigskip\noindent
    \emph{MathSciNet codes:}\quad 05B05, 20B25, 05B25

    \bigskip\noindent
    This work forms part of the Australian Research Council Discovery Grant project DP200100080.
    
\end{abstract}

\maketitle

\section{Introduction}

We present several constructions of block-transitive $2$-designs to demonstrate new simple ways of `blowing up’ small $2$-designs to form larger ones that preserve certain additional structure on the point set. Our approach may be viewed as a generalisation of the Cameron—Praeger design construction in \cite{CP93} which in turn was inspired by the work of Delandtsheer and Doyen \cite{DD} first brought to the attention of Cameron and the third author by Alan Camina, to whom this paper is dedicated. 

The most general structure we have previously considered on the point set of a block-transitive $2$-design is  a poset of partitions   $(\CC,\preccurlyeq)$, and in \cite[Theorem 1.2]{SmallEx} we gave (theoretical) necessary and sufficient conditions on the parameters for an orbit of a $k$-subset of points under the full stabiliser of  $(\CC,\preccurlyeq)$ to be a $2$-design. In \cite{SmallEx} we also presented infinite families of $2$-designs corresponding to several small posets. However we did not show, for an  arbitrary poset, whether or not the conditions for a $2$-design could be satisfied. The aim of this paper is to present a number of general constructions, as well as several explicit infinite families of block-transitive $2$-designs corresponding to families of posets of unbounded size for which, to our knowledge, explicit design constructions were not previously known.  We note that the necessary and sufficient conditions given in \cite[Theorem 1.2]{SmallEx} could have been used to justify the validity of our constructions, and indeed these were used in our first analyses of the constructions. However we subsequently developed a more straightforward approach involving simpler and shorter proofs, and giving more insight into the structure of the designs. This simpler approach is presented in the paper. In Section~\ref{s:compare} we briefly discuss how our constructions fit within the general framework presented in \cite{SmallEx}.

To explain the context of these design constructions, we note that the Cameron—Praeger designs in \cite{CP93} are general constructions for a poset consisting of a single nontrivial point partition. Also general constructions for posets which are either a chain or an anti-chain can be found in \cite{multigrids, grids22, chainspaper, ftchainspaper}. Combined with  constructions presented in \cite[Theorem 1.3]{SmallEx} we have explicit constructions  for block-transitive $2$-designs corresponding to each poset $(\CC,\preccurlyeq)$ of size up to $3$, as well as the famous `N-poset' (with four nodes). The constructions in this paper are to our knowledge the first ones corresponding to posets of unbounded size apart from chains. Our strategy is to give two general design constructions, and then apply each recursively to produce explicit families of block-transitive $2$-designs for posets of unbounded size.

Our first   Construction~\ref{con:wr-comb} takes as input a $1$-design $\widetilde{\D} = (\widetilde{\PP},\widetilde{\B})$ and an $e$-element set $\Delta$ for some integer $e>1$, and produces a $1$-design $\D(\widetilde{\D},e)$ with point set $\widetilde{\PP}\times \Delta$. The construction can produce a $2$-design from a given $\widetilde{\D}$ if and only if $\widetilde{\D}$ itself is a $2$-$(\widetilde{v},\widetilde{k},\widetilde{\lambda})$ design such that the number of blocks divides $\widetilde{v}\cdot\widetilde{\lambda}$, and in this case there is a unique value of $e$ for which  $\D(\widetilde{\D},e)$ is a $2$-design (see Proposition~\ref{prop:con-wr-comb} and \eqref{eq:tformula}). For example, all symmetric $2$-designs $\widetilde{\D}$ have this special divisibility property, as do all the designs in \cite[Constructions 7.4, 7.1]{multigrids} and \cite[Examples 5.1--5.4]{SmallEx}, and several other infinite families of designs, see our discussion in Remark~\ref{r:con-wr-ex}. % {\color{red} remark should be adjacent to Example \ref{ex:chains}.}
If an input $2$-design $\widetilde{\D}$ for Construction~\ref{con:wr-comb} has  the divisibility property and admits a block-transitive automorphism group $G$ preserving a poset of partitions of the point set based on a poset $\widetilde{\I} = (\widetilde{I},\preccurlyeq)$, then the output design $\D(\widetilde{\D},e)$  admits the block-transitive poset-imprimitive group $G\wr S_e$ corresponding to the poset obtained from $\widetilde{\I} = (\widetilde{I},\preccurlyeq)$ by adding a single new point $s$ such that $i\preccurlyeq s$ for all $i\in \widetilde{I}$, as shown in Figure~\ref{fig:firstI}, see Remark~\ref{rem:poset-con:wr-comb}. For instance, beginning with a $2$-design corresponding to the grid poset $\widetilde{\I}$ consisting of two unrelated points, we obtain a $2$-design  corresponding to the inverted V-shaped poset, see Example~\ref{ex:V-inv}.

\begin{figure}
    \centering
    \includegraphics{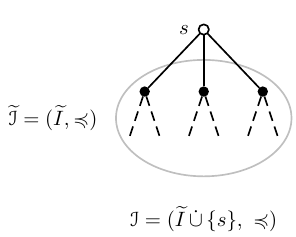}
    \caption{Poset after one application of Construction~\ref{con:wr-comb}}
    \label{fig:firstI}
\end{figure}

If the input design to Construction~\ref{con:wr-comb} is a poset-imprimitive $2$-design $\widetilde{\D}$ corresponding to a poset $\widetilde{\I} = (\widetilde{I},\preccurlyeq)$, and satisfies certain conditions on its parameters, then we may apply Construction~\ref{con:wr-comb} 
recursively $n$ times, as described in Construction~\ref{cons-iterated} and proved in Theorem~\ref{thm:addchainontop}, to produce a new block-transitive, poset-imprimitive  $2$-design corresponding to the poset obtained by adding a chain of length $n$ on top of $\widetilde{\I}$. Several families of such posets of unbounded size are shown in Figures~\ref{fig:invertedY} and~\ref{fig:smallposets}. In Examples~\ref{ex:invertedY} and~\ref{ex:smallposets} we obtain explicit families corresponding to these posets. This already achieves our stated aims, but we offer in addition a second design-construction strategy.
% {\color{blue}
% ...}        

Our second   Construction~\ref{con:rect-comb} also takes as input a $1$-design $\widetilde{\D} = (\widetilde{\PP},\widetilde{\B})$, but this time the second ingredient is a Cartesian product $\Delta=\Delta_1\times \Delta_2$ with each $|\Delta_i|=e_i\geq2$ such that $e_1-1$ divides $e_2-1$.  The output is again a $1$-design $\D(\widetilde{\D},(e_1,e_2))$, and is shown in   Proposition~\ref{prop:con-rect-comb} to be a $2$-design if and only if $\widetilde{\D}$ itself is a $2$-design and the conditions on parameters in \eqref{hyp:con-rect} hold. Again we construct new block-transitive poset-imprimitive $2$-designs for unboundedly large posets, namely the $Y$-shaped posets in Figure~\ref{ex:Y}, see Example~\ref{ex:Y}. In addition Construction~\ref{con:rect-comb} yields, in Example~\ref{ex:smallposets-2nodesontop}, explicit infinite families of $2$-designs corresponding to each of the small posets in Figure~\ref{fig:smallposets-2nodesontop}, using as input design $\widetilde{\D}$ some of the $2$-designs constructed in \cite{multigrids} and \cite{SmallEx}.

In the final Section~\ref{s:other} we %celebrate 
capitalise on the flexibility  available to us in applying Constructions~\ref{con:wr-comb} and~\ref{con:rect-comb} recursively, in any order, and obtain new block-transitive 2-designs with groups preserving poset block structures for a large variety of posets, including those shown in Figure~\ref{fig:towers}. Taking into account these new constructions, together with previous work by the authors and some upcoming work by Dacaymat \cite{simplecond}, we now have explicit constructions for infinite families of block-transitive $2$-designs corresponding to all but five of the posets with four-nodes, namely the posets shown in Figure~\ref{fig:4nodes-3}. We discuss the available known constructions for small posets in Example~\ref{ex:4nodes}, and pose an open problem.

% \bigskip

% In this paper, we give  methods for ``extending" a block-transitive, $\widetilde{\I}$-imprimitive $2$-design $\widetilde{\D} = (\widetilde{\PP},\widetilde{B})$ based on a poset $\widetilde{\I} = (\widetilde{I},\preccurlyeq)$ to a block-transitive, poset-imprimitive $2$-design $\D = (\PP,\B)$ based on a poset $(I,\preccurlyeq)$ which contains $(\widetilde{I},\preccurlyeq)$ as a sub-poset. 

\section{Preliminaries}

A \emph{block design} is an incidence structure $\D = (\PP,\B)$ consisting of a finite set $\PP$ of \emph{points} and a collection of point-subsets called \emph{blocks}, such that, for some constants $k\geq 2,r \geq 1$, every block has size $k$ and every point appears in exactly $r$ blocks. If the point set $\PP$ has size $v$, and if, in addition, for some constants $u,\lambda > 0$ any $u$-subset of points appears in $\lambda$ blocks, then the design is a $u$-$(v,k,\lambda)$ design. When $k=2$ a block design is simply a graph, and the only such 2-designs are the complete graphs.  Also, when $r=1$ the block set of a block design is simply a partition of the point set, and the only 2-designs with $r=1$ are designs consisting of one single block of size $v=k$. Such a design is called {\it trivial}. When every $k$-set is a block, the design is called {\it complete}; it is a $k$-design and hence a $2$-design. In particular, each trivial $2$-design is a complete design.  

By counting flags (incident point-block pairs)  in a block design, we find that:
    \begin{equation} \label{eq:vr=bk}
    vr= bk.
    \end{equation}
If the block design is a $2$-$(v,k,\lambda)$ design then, counting the number of flags $(p,B)$ such that $p_0\in B$ and $p\neq p_0$, where $p_0$ is any given point, we see that
    \begin{equation} \label{eq:r(k-1)}
    (v-1)\lambda=r(k-1),\quad \text{and hence by \eqref{eq:vr=bk} that}\quad v(v-1)\lambda=bk(k-1).
    \end{equation}
Our results in later sections use the parameter $t=\frac{k(k-1)}{v-1}$ for a $2$-design. Using \eqref{eq:r(k-1)}, we deduce that in a $2$-$(v,k,\lambda)$ design,       
    \begin{equation} \label{eq:tformula}
    t= \frac{k(k-1)}{v-1} = \frac{k\lambda}{r} = \frac{v\lambda}{b}.
    \end{equation}

Moreover, in our examples and discussion we often use notation from \cite{SmallEx}.

\section{Comparison of our new approach with that in \cite{SmallEx}} \label{s:compare}

To explain our approach to design construction and analysis used in \cite{SmallEx} we need to introduce the notion of an ancestral subset of a poset (in Subsection~\ref{ss:ancestral}), and of a poset block structure and generalised wreath products preserving such structures (in Subsection~\ref{ss:posetblockstruc}). We then describe, and comment on, the criterion given in \cite{SmallEx} for a  generalised wreath product to be poset-imprimitive on points and block-transitive on a $2$-design based on a poset block structure (in Subsection~\ref{ss:poset-imp}).  

\subsection{Partially ordered sets and their ancestral subsets} \label{ss:ancestral}

Let $\I = (I,\preccurlyeq)$ be a partially ordered set, and for any two elements $i, j \in I$ write $i \prec j$ if $i \preccurlyeq j$ but $i \neq j$. We may visualise $\I$ using a node for each element of $I$, such that if $i \prec j$, then node $i$ is `below' node $j$ and $i$ and $j$ are connected by an edge. For example, $\I$ is a \emph{chain} if any two distinct elements are related, and $\I$ is an \emph{antichain} if no two distinct elements are related. A chain with $s$ elements is represented by a vertical line of $s$ nodes, while an antichain with $s$ elements is represented by a set of $s$ isolated nodes.

A subset $J \subseteq I$ is said to be \emph{ancestral} if for every $j \in J$ and any $i \succcurlyeq j$, we have $i \in J$. For example, the empty subset $\varnothing$ and the improper subset $I$ are both ancestral in any poset $\I$, and the union of two ancestral subsets is an ancestral subset. Moreover each $i \in I$ determines a special ancestral subset $A[i]$ defined as
    \begin{equation}
    A[i] := \left\{ j \in I \ \vline \ j \succcurlyeq i \right\},
    \end{equation}
and if $i \prec j$ then $A[j] \subset A[i]$. Clearly $i \in A[i]$, and each ancestral ancestral subset $J$ is the union of the sets $A[i]$ for  $i \in J$. Also  the \emph{border} $\partial J$ of an ancestral subset $J$ is the set of all maximal elements in the complement $J^{\c} = I \setminus J$ of $J$ in $I$, and  for any $S \subseteq \partial J$, the set $J \cup S$ is an ancestral subset.

For example, in a chain $(\{1, \ldots, s\}, \ \preccurlyeq)$ with $1 \prec \cdots \prec s$, the nonempty ancestral subsets are precisely the sets $A[i]$ for  $i \in I$, and $\partial A[i] = \{i-1\}$ for $i>1$. On the other hand, in an antichain, $A[i] = \{i\}$ for all $i \in I$, and $\partial J = J^{\c}$ for each ancestral subset $J$.

% {\color{red} Do we need the next notation?}
% Let $\mathscr{A}(\I)$ denote the family of ancestral subsets of $\I = (I,\preccurlyeq)$.

\subsection{Poset block structures and generalised wreath products} \label{ss:posetblockstruc}

Let $\I = (I,\preccurlyeq)$ be a finite partially ordered set. We obtain as follows a point set $\PP$ together with a set $\bm{\C}^*$ of partitions of $\PP$ that form a partially ordered set isomorphic to $\I$. We call the pair $(\PP, \bm{\C}^*)$ a \emph{poset block structure}.

\medskip\noindent
\emph{The point set $\PP$}:\quad 
For each $i \in I$ choose a set $\Delta_i$ of size $e_i:=|\Delta_i| > 1$, and set $\PP = \prod_{i \in I} \Delta_i$, so that $v:=|\PP|=\prod_{i\in I}e_i$. For each subset $J \subseteq I$, let $\Delta_J = \prod_{j \in J} \Delta_j$, with the convention that $\Delta_\varnothing$ is a singleton, and note that $\PP = \Delta_I$.

\medskip\noindent
\emph{The partition $\C_J$ corresponding to an ancestral subset  $J\subseteq I$}:\quad 
Let $\pi_J$ be the natural projection $\pi_J : \PP \rightarrow \Delta_J$, and define
    \begin{equation} \label{partition}
    \C_J = \big\{ C_{\bm{\nu}} \ | \ {\bm{\nu}} \in \Delta_J \big\} \quad 
    \text{where} \quad
    C_{\bm{\nu}} = \big\{ \bm{\delta} \in \PP \ | \ (\bm{\delta})\pi_J = \bm{\nu} \big\}.
    \end{equation}
Then $\C_J$ is a partition of $\PP$ with classes $C_{\bm{\nu}}$. In particular, $\C_I = \binom{\PP}{1}$, the partition into singletons, and (since $\Delta_\varnothing$ is a singleton) $\C_\varnothing = \{\PP\}$.

\medskip\noindent
\emph{The first partition poset  $\bm\C$}:\quad Let  $\bm\C$ be the set of partitions $\C_J$ for all ancestral subsets  $J\subseteq I$. 
For ancestral subsets such that $J \subseteq J'$,  the partition $\C_{J'}$ is a refinement of $\C_J$, that is to say, each class in $\C_{J'}$ is a subset of a class in $\C_J$. We use the symbol $\preccurlyeq$ also to denote refinement, that is, we write ``$\C_{J'}$ is a refinement of $\C_J$'' as $\C_{J'} \preccurlyeq \C_J$; and $\C_{J'} \prec \C_J$ if $\C_{J'}$ is a proper refinement of $\C_J$ (some class in $\C_{J'}$ is a proper subset of a class in $\C_J$). Then $\bm{\C}$ is a partially ordered set  under this relation $\preccurlyeq$.

% Let
%     \begin{equation*}
%     \bm{\C} = \big\{ \C_J \ \vline \ J\in \mathscr{A}(\I) %\subseteq I \ \text{is ancestral} 
%     \big\}.
%     \end{equation*} 
% The relation $\preccurlyeq$ is a partial order on the set of all $\PP$-partitions in $\bm{\C}$, and the pair $(\PP,\bm{\C})$ is called a \emph{poset block structure}.

\medskip\noindent
\emph{The partition sub-poset  $(\bm{\C}^*,\preccurlyeq)$ isomorphic to $\I$}:\quad 
%
% Observe that for any ancestral subsets $J$ and $J'$, the partition $\C_{J \cup J'}$ consists of all classes $C \cap C'$ where $C \in \C_J$ and $C' \in \C_{J'}$.
%
Define $\bm{\C}^* = \big\{ \C_{A[i]} \ \vline \ i \in I \big\}$, and recall that,  if $i \prec j$ then $A[i] \supset A[j]$, so $\C_{A[i]} \prec \C_{A[j]}$. Hence $(\bm{\C}^*, \preccurlyeq) \cong \I$, and $(\PP, \bm{\C}^*)$ is called a \emph{poset block structure}. In general $\bm{\C}^*$ may be a proper subset of $\bm{\C}$ (for examples see \cite[Section 2.2, especially Figure 2]{SmallEx}). Nevertheless 
for each ancestral subset $J$, that is, for each $\C_J\in\bm{\C}$, the partition $\C_J$ is the meet of the partitions $\C_{A[i]}\in \bm{\C}^*$ for $i\in J$ 
(see the discussion after display (10) in \cite[Section 2.2]{SmallEx}).

\medskip\noindent
\emph{Generalised wreath products associated with   $(\bm{\C}^*,\preccurlyeq)$}:\quad 
Let $(\PP, \bm{\C}^*)$ be a poset block structure with partition poset $(\bm{\C}^*,\preccurlyeq)$ isomorphic to the poset $\I=(I,\preccurlyeq)$, and $\PP = \prod_{i \in I} \Delta_i$ of size $v=\prod_{i\in I} e_i$ as above. The \emph{generalised wreath product} $F$ of subgroups $G_i\leq \Sym(\Delta_i)$, for $i\in I$, is denoted $F = \prod_{(I,\preccurlyeq)} (G_i,\Delta_i)$,  can be written as a product $F = \prod_{i \in I} G_i^{|\Delta_{A(i)}|}$, and has a faithful action on $\PP$, so $F$ may be viewed as a subgroup of $\Sym(\PP)$, see \cite[Section 3]{BPRS} and \cite[Section 2.3]{SmallEx} for a detailed discussion of this group and how it acts on $\PP$. Moreover $F$ leaves invariant each partition in $\CC^*$, and if  $G_i$ is transitive on $\Delta_i$ for each $i\in I$, then $F$ is transitive on $\PP$, see  \cite[Lemma 9]{BPRS}. 
The largest subgroup of $\Sym(\PP)$ that leaves invariant this poset block structure is $G= \prod_{(I,\preccurlyeq)} (\Sym(\Delta_i),\Delta_i)$ by \cite[Theorem B]{BPRS}. Moreover, $F$ has the same orbitals as $G$ (and hence the same orbits on pairs of points) if and only if $G_i$ is 2-transitive on $\Delta_i$ for each $i\in I$, by \cite[Theorem C]{BPRS}. 

% classes in  $\C_J$ are intersections $\bigcap_{j \in J} C_j$ of classes $C_j \in \C_{A[j]}$ for $j \in J$.

% the subset $\bm{\C}^*$ of $\bm{\C}$ by
%     \begin{equation*}
%     \bm{\C}^* = \big\{ \C_{A[i]} \ \vline \ i \in I \big\}.
%     \end{equation*}
% Recall that if $i \prec j$ then $A[i] \supset A[j]$, so $\C_{A[i]} \prec \C_{A[j]}$. Hence $(\bm{\C}^*, \preccurlyeq) \cong \I$. Note that $\bm{\C}^* \subseteq \bm{\C}$, but 

% in general $\bm{\C}^* \neq \bm{\C}$, that is, $\bm{\C}$ may contain partitions other than those in $\bm{\C}^*$. Since every ancestral subset $J$ is the union of subsets $A[i]$, every partition in $\bm{\C}$ can be obtained by taking intersections of classes of some partitions $\C_{A[i]}$. In particular, the classes in the partition $\C_J$ are the intersections $\bigcap_{j \in J} C_j$ of classes $C_j \in \C_{A[j]}$ for all $j \in J$.

% \subsection{Generalised wreath product} \label{ss:genwr}

% Let $\I = (I,\preccurlyeq)$ be a partially ordered set, and let $\PP = \prod_{i \in I} \Delta_i$ be a set of points with poset block structure $(\CC,\preccurlyeq)$, $\CC = \{ \C_J \ | \ J \in \A(\I) \}$, as described in Subsection \ref{ss:posetblockstruc}. For each $i$ let $G_i \leq \Sym(\Delta_i)$. The \emph{generalised wreath product} $F$ of the groups $G_i$ is {\color{blue} $F = \prod_{i \in I} G_i^{|\Delta_{A(i)}|}\leq\Sym(\PP)$, see \cite[Section 3]{BPRS} and \cite[Section 2.3]{SmallEx} for a detailed discussion of this group and how it acts on $\PP$. In particular, $F$ preserves each partition in $\CC$. }

\subsection{$\I$-imprimitive designs}\label{ss:poset-imp}

Let $\I=(I,\preccurlyeq)$ with $|I|\geq 2$, let $\PP = \prod_{i \in I} \Delta_i$  of size $v = \prod_{i \in I} e_i$, for each $i\in I$ let $G_i \leq \Sym(\Delta_i)$ acting $2$-transitively on $\Delta_i$, and  $F = \prod_{(I,\preccurlyeq)} (G_i,\Delta_i)$, (all as in Subsection~\ref{ss:posetblockstruc}). Also let $B \subset \PP$ of size $k$, and let $\B = B^F$ be the set of images of $B$ under elements of $F$. The main theorem \cite[Theorem 1.2]{SmallEx} gives precise conditions under which $\D = (\PP,\B)$ is a $2$-design; namely,  $\D = (\PP,\B)$ is a $2$-design if and only if, for each proper nonempty ancestral subset $J\subset I$, the following equation holds (using the terminology above).
    \begin{equation} \label{2des}
    \sum_{S \subseteq \Bdy{J}}(-1)^{|S|} \left(\sum_{\bm{\nu} \in \Delta_{J \cup S}}  \big|B \cap C_{\bm{\nu}}\big|^2\right) 
    = \frac{k(k-1)}{v-1} \left( \prod_{i \in \Bdy{J}} (e_i - 1) \right) \left( \prod_{j \in (J \cup \Bdy{J})^\c} e_j \right).
    \end{equation} 
This was the first necessary and sufficient criterion given for $2$-designs admitting a poset-imprimitive, block-transitive group of automorphisms, for a general finite poset $\I$. 
We observe that the number of equations to be checked is two less than the number of ancestral subsets of $\I$; and also that these equations involve an alternating sum. Recently an equivalent criterion was discovered  which removes the complication of alternating sums,  see \cite{simplecond}. 
Our aim in this paper has been to provide a much simpler, if less general, approach, by restricting to a proper (but still infinite) family of posets where simple recursive procedures could be developed which, in our view, makes the criteria much more intuitive and transparent.

\section{Construction: adding a single node on top of the poset} %: a combinatorial approach}

We begin with a very general construction (Construction~\ref{con:wr-comb}) of a $1$-design which can produce $2$-designs under appropriate conditions on the parameters (Proposition~\ref{prop:con-wr-comb}). We show that recursively applying Construction~\ref{con:wr-comb} is often possible to produce new families of block-transitive poset-imprimitive $2$-designs (Theorem~\ref{thm:addchainontop}). We give explicit examples of several infinite families of (new and old) $2$-designs arising from this approach.

\begin{construction} \label{con:wr-comb}
Let  $\widetilde{\D} = (\widetilde{\PP},\widetilde{\B})$ be a $1$-design where $\widetilde{v} := |\widetilde{\PP}|$ and each block has size $\widetilde{k}\geq 2$. The number of blocks of $\widetilde{\D}$ is denoted by $\widetilde{b}$, and the number of blocks through a point is denoted by $\widetilde{r}$. Let $e$ be an integer such that $e>1$, let $\Delta:=\{1,2,\ldots, e\}$, and define 
    \[ \PP = \widetilde{\PP}\times \Delta. \]
We define $\B = \bigcup_{i\in\Delta} \B_i$  as follows (blocks in the family $\B_i$ are depicted in Figure \ref{fig:con-wr-comb}): For each  $i \in \Delta$, a subset $B \subset \PP$ is in $\B_i$ if and only if
    % \[
    % \left\{ \begin{array}{l}
    % B \cap (\widetilde{\PP} \times \{i\}) = \widetilde{B} \times \{i\} \text{ for some } \widetilde{B} \in \widetilde{\B}, \ \text{and} \\
    % |B \cap (\widetilde{\PP} \times \{j\})| = 1 \text{ if $j \neq i$}.
    % \end{array}\right.
    % \]
    \begin{enumerate}
    \item[(a)]  $B \cap (\widetilde{\PP} \times \{i\}) = \widetilde{B} \times \{i\}$ for some $\widetilde{B} \in \widetilde{\B},$ and 
    \item[(b)] $|B \cap (\widetilde{\PP} \times \{j\})| = 1 \text{ if $j \neq i$}.$
    \end{enumerate}    
Let $\D(\widetilde{\D},e) := (\PP,\B)$.

    \begin{center}
    \begin{figure}
    \centering
    \includegraphics{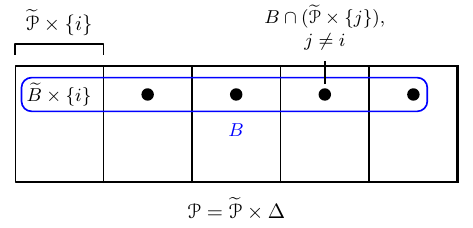}
    \caption{Block $B \in \B_i$ in Construction \ref{con:wr-comb}}
    \label{fig:con-wr-comb}
    \end{figure}
    \end{center}
\end{construction}

First we make some comments on the parameters and symmetry of the incidence structures produced by Construction~\ref{con:wr-comb}.

\begin{remark}\label{remark-combcons}
In Construction \ref{con:wr-comb}, we see that the set $\PP$ has size $v = \widetilde{v} \cdot e$, each $B \in \B$ has size $k := |B| = \widetilde{k} + e - 1 $, and the total number of blocks is $|\B| = e \cdot \widetilde{b} \cdot \widetilde{v}^{e-1}$.
We note that the design $\D(\widetilde{\D},e)$ is always a $1$-design since, using \eqref{eq:vr=bk} for  $\widetilde{\D}$, the number of blocks containing a point of $ \widetilde{\PP} \times \Delta$ is 
    \[
    r
    = \widetilde{r} \cdot \widetilde{v}^{e-1} + (e-1) \cdot \widetilde{b} \cdot \widetilde{v}^{e-2}
    = \widetilde{v}^{e-2} \cdot \widetilde{b}(\widetilde{k} + e - 1)
    = \widetilde{v}^{e-2} \cdot \widetilde{b} \cdot k
    \]
It is easy to see that $\Aut(\widetilde{\D}) \wr S_e \leqslant \Aut(\D(\widetilde{\D},e))$. Moreover, if $\widetilde{G} \leqslant \Aut(\widetilde{\D})$ is transitive on $\widetilde{\B}$, then $\widetilde{G} \wr S_e$ is transitive on $\B$. 
\end{remark}

\begin{proposition} \label{prop:con-wr-comb}
The point-block structure $\D(\widetilde{\D},e)$ in Construction \ref{con:wr-comb} is a $2$-design if and only if 
    \begin{equation}\label{hyp:con}
    \widetilde{\D} \text{ is a $2$-design}, \ t := \frac{\widetilde{k}(\widetilde{k}-1)}{{\widetilde{v}}-1} \text{ is an integer and } e = ({\widetilde{k}}-1)({\widetilde{k}}-2) + t.
    \end{equation}
% In that case, if $\widetilde{\D}$ is a $2$-$(\widetilde{v}, \widetilde{k}, \widetilde{\lambda})$ design, 
Moreover, if \eqref{hyp:con} holds and $\widetilde{\D}$ is a $2$-$(\widetilde{v},\widetilde{k},\widetilde{\lambda})$ design, then $\D(\widetilde{\D},e)$ is a $2$-$(v,k,\lambda)$ design with parameters 
    \begin{align*}
    v &= \widetilde{v}\cdot e=\widetilde{v}\cdot(({\widetilde{k}}-1)({\widetilde{k}}-2)+t)\\
    k &= \widetilde{k} + e - 1=(\widetilde{k}-1)^2+t\\
    \lambda &= \widetilde{\lambda} \cdot \widetilde{v}^{e - 1}
    \end{align*}
and
% Moreover, if \eqref{hyp:con} holds,
    \begin{equation} \label{e:kv} 
    \frac{k(k-1)}{v-1} = t.
    \end{equation}
\end{proposition}

\begin{proof}
For two points $p,q \in \widetilde{\PP}$, let $\widetilde{\lambda}(p,q)$ be the number of blocks in $\widetilde{\B}$ containing $p$ and $q$.

Consider first two points $(p,i)$ and $(q,i)$ in the same set $\widetilde{\PP}\times\{i\}$. For $j \neq i$, blocks in $\B_j$ contain only one point in $\widetilde{\PP}\times\{i\}$. So all blocks containing $(p,i)$ and $(q,i)$ are in $\B_i$. The number of such blocks is 
    \begin{equation} \label{eq:sameclass}
    \widetilde{\lambda}(p,q) \cdot \widetilde{v}^{e-1}.  
    \end{equation}

Now we consider two points $(p,i)$ and $(q,j)$ where $i\neq j$. The number of blocks in $\B_i$ containing $(p,i)$ and $(q,j)$ is $\widetilde{r} \cdot \widetilde{v}^{e-2}$, which by \eqref{eq:vr=bk} is equal to $\widetilde{b}\cdot \widetilde{k} \cdot \widetilde{v}^{e-3}$.
The number of such blocks in $\B_j$ is the same. For $\ell \notin \{i,j\}$, the number of such blocks in $\B_\ell$ is $ \widetilde{b} \cdot \widetilde{v}^{e-3}$. Thus the total number of blocks in $\B$ containing $(p,i)$ and $(q,j)$ is 
    \begin{equation} \label{eq:diffclass}
    2  \widetilde{b}\cdot \widetilde{k} \cdot \widetilde{v}^{e-3} + (e-2) \cdot \widetilde{b} \cdot \widetilde{v}^{e-3}=\widetilde{b} \cdot \widetilde{v}^{e-3}\cdot(2\widetilde{k}+e-2)
    \end{equation}

Thus $\D(\widetilde{\D},e)$ is a $2$-design if and only if 
    \[\widetilde{b} \cdot \widetilde{v}^{e-3}\cdot(2\widetilde{k}+e-2)= \widetilde{\lambda}(p,q) \cdot \widetilde{v}^{e-1} \text{ for any } p,q \in \widetilde{\PP},\]
that is, if and only if 
    \begin{equation*} %\label{eq:2design-comb}
    \text{ $\widetilde{\D}$ is a $2$-$(\widetilde{v}, \widetilde{k}, \widetilde{\lambda})$ design  and }
   \widetilde{b}\cdot(2\widetilde{k}+e-2) = \widetilde{\lambda} \cdot \widetilde{v}^{2}.
    \end{equation*}

Thus $\D(\widetilde{\D},e)$ is a $2$-design if and only if 
    \begin{equation} \label{eq:2design-comb}
    \text{$\widetilde{\D}$ is a $2$-$(\widetilde{v}, \widetilde{k}, \widetilde{\lambda})$ design  and }  \widetilde{b}\cdot (2\widetilde{k} + e - 2) = \widetilde{\lambda} \cdot \widetilde{v}^2.
    \end{equation}
By \eqref{eq:r(k-1)} and \eqref{eq:vr=bk},
    \[ \widetilde{\lambda}
    = \frac{\widetilde{r} \cdot (\widetilde{k}-1)}{\widetilde{v}-1}
    = \frac{\widetilde{b} \cdot \widetilde{k} \cdot (\widetilde{k}-1)}{\widetilde{v} \cdot (\widetilde{v}-1)}. \]
Thus the identity in \eqref{eq:2design-comb} becomes 
    \[
    \frac{\widetilde{k} \cdot (\widetilde{k}-1)}{(\widetilde{v}-1)}\cdot\widetilde{v} = 2\widetilde{k} + e - 2
    \]
In other words, 
    \begin{align*} 
    e
    &= \frac{\widetilde{k} \cdot (\widetilde{k}-1)}{(\widetilde{v}-1)} \cdot \widetilde{v} - 2\widetilde{k} + 2 \\
    &= \frac{\widetilde{k} \cdot (\widetilde{k}-1)}{(\widetilde{v}-1)} + \widetilde{k} \cdot (\widetilde{k}-1) - 2\widetilde{k} + 2 \\
    &= \frac{\widetilde{k} \cdot (\widetilde{k}-1)}{(\widetilde{v}-1)} + (\widetilde{k}-1) \cdot (\widetilde{k}-2)
    \end{align*}
Since $e$ must be an integer, $\D(\widetilde{\D},e)$ is a $2$-design if and only if $\widetilde{\D}$ is a $2$-design, $t = \frac{\widetilde{k} \cdot (\widetilde{k}-1)}{(\widetilde{v}-1)}$ is an integer and $e = t + (\widetilde{k}-1) \cdot (\widetilde{k}-2)$. Moreover, if these conditions hold, then $\lambda = \widetilde{\lambda} \cdot \widetilde{v}^{e-1}$ as computed in \eqref{eq:sameclass} and
    \[ k = \widetilde{k} + e - 1 = \widetilde{k} + t + (\widetilde{k}-1) \cdot (\widetilde{k}-2) - 1 = (\widetilde{k}-1)^2 + t.\]

Finally, assuming \eqref{hyp:con} holds, we compute
    \begin{align*}
    k(k-1)&=(({\widetilde{k}}-1)^2+t)(({\widetilde{k}}-1)^2+t-1)\\
    &=({\widetilde{k}}-1)^2\left( ({\widetilde{k}}-1)^2+2t-1 \right)+t^2-t\\
    &=({\widetilde{k}}-1)^2\left( \widetilde{k}^2-2{\widetilde{k}}+2t \right)+t^2-t
    \end{align*}
while
    \begin{align*}
    t(v-1)&=t(\widetilde{v} e-1)=t\left(e(\widetilde{v}-1)+e-1\right)\\
    &=e {\widetilde{k}}({\widetilde{k}}-1)+t(e-1)\\
    &={\widetilde{k}}({\widetilde{k}}-1)+(e-1)({\widetilde{k}}({\widetilde{k}}-1)+t)\\
    &={\widetilde{k}}({\widetilde{k}}-1)+(({\widetilde{k}}-1)({\widetilde{k}}-2)+t-1)({\widetilde{k}}({\widetilde{k}}-1)+t)\\
    &={\widetilde{k}}({\widetilde{k}}-1)+({\widetilde{k}}-1)\left({\widetilde{k}}({\widetilde{k}}-1)({\widetilde{k}}-2)+t({\widetilde{k}}-2)+(t-1){\widetilde{k}} \right)  +t^2-t\\
    &=({\widetilde{k}}-1)\left({\widetilde{k}}+{\widetilde{k}}({\widetilde{k}}-1)({\widetilde{k}}-2)+t{\widetilde{k}}-2t+t{\widetilde{k}}-{\widetilde{k}} \right)  +t^2-t\\
    &=({\widetilde{k}}-1)\left({\widetilde{k}}({\widetilde{k}}-1)({\widetilde{k}}-2)+2t({\widetilde{k}}-1) \right)  +t^2-t\\
    &=({\widetilde{k}}-1)^2 \left( {\widetilde{k}}({\widetilde{k}}-2) + 2t \right) + t^2 - t.
    \end{align*}
So the two values are equal.
 \end{proof}

We make a several remarks about Proposition~\ref{prop:con-wr-comb}. The first concerns the smallest designs produced by Construction~\ref{con:wr-comb}.

\begin{remark}
(a) We note that in Construction~\ref{con:wr-comb} the parameters satisfy $e\geq2$ and $\widetilde{k}\geq 2$, and in \eqref{hyp:con} the integer $t$ is always positive. The expression in \eqref{hyp:con} for $e$ then yields (since  $e\geq 2$) that either $\widetilde{k}\geq3$ or $t\geq2$. Moreover, if $\widetilde{k}\geq3$ then $e\geq 3$.

\medskip\noindent(b)  On the other hand if $\widetilde{k}=2$, then $e=t$ so we must have $t\geq 2$ and the expressions in \eqref{hyp:con} yield $(\widetilde{v}, \widetilde{k},t,e)=(2,2,2,2)$. These values correspond to a unique  $2$-design in Proposition~\ref{prop:con-wr-comb}, namely for $\widetilde{\D}$ the trivial $2$-$(2,2,1)$ design we obtain as $\D(\widetilde{\D},2)$ the complete  $2$-$(4,3,2)$ design. 

\medskip\noindent
(c) For the design  $\D(\widetilde{\D},2)$ described in (b), we have  $S_2 \wr S_2\leqslant \Aut(\D)$ by Remark~\ref{remark-combcons}. However,  the full automorphism group $\Aut(\D)$ is larger: namely  $\Aut(\D) = S_4$ (since $\D$ is a complete design) which is point-primitive. 

\medskip\noindent
(d) It seems to us unusual that the automorphism group of the $2$-design in part (b) is larger than the group  $\Aut(\widetilde{\D}) \wr S_{e}$ given by Remark~\ref{remark-combcons}.
Computations with Magma \cite{magma} on lots of small examples point to 
$\Aut(\D)=\Aut(\widetilde{\D})\wr S_{e}$ most of the time. The only  exceptions we found are the design $\D(\widetilde{\D},2)$ described in (b), and the design $\D(\widetilde{\D},2)$ where $\widetilde{\D}$ is a $C_4$-graph (in this case neither $\widetilde{\D}$ nor $\D(\widetilde{\D},2)$ is a $2$-design).
We wonder if these two examples are exceptional, and we believe:
\begin{conjecture}
     If $e\geq 3$ or $\widetilde{k}\geq 3$, then $\Aut(\D(\widetilde{\D}, e) )= \Aut(\widetilde{\D}) \wr S_{e}$. 
\end{conjecture}

\noindent
Note that, even for $e=\widetilde{k}=2$, the equality is true in many cases, according to our computations.
\end{remark}

Our next remark discusses the variety of input $2$-designs  $\widetilde{\D}$ which can yield  $2$-designs $\D$ in Construction~\ref{con:wr-comb}.

\begin{remark}\label{r:con-wr-ex}
% [$2$-designs with integer $t$ - response to comments after Ex. \ref{ex:chains}]
By Proposition~\ref{prop:con-wr-comb}, in order for   Construction~\ref{con:wr-comb} to produce a $2$-design we need, as the input design $\widetilde{\D}$, a $2$-$(\widetilde{v}, \widetilde{k}, \widetilde{\lambda})$ design  for which the parameter $t = \frac{\widetilde{k}(\widetilde{k}-1)}{\widetilde{v}-1}$ is an integer.
We list below several infinite families of designs with this property. Each of these can then be used as  input to Construction \ref{con:wr-comb} to produce a $2$-design $\D$ for the appropriate value of $e$.
    \begin{enumerate}[(a)]
        \item $\widetilde{\D}$ is a trivial design, that is, $\widetilde{k} = \widetilde{v}$ so that $t = \widetilde{k}$.
        \item $\widetilde{\D}$ is a complete design with $\widetilde{k} = \widetilde{v}-1$, so that $t = \widetilde{k}-1$.
        \item $\widetilde{\D}$ is a symmetric $2$-design, that is $\widetilde{v}  = \widetilde{b}$, so that $t = \widetilde{\lambda}$ by \eqref{eq:tformula} (for example, any of the $2$-transitive examples classified by Kantor~\cite{K}; these comprise the design of points and hyperplanes of  a Desarguesian projective space ${\rm PG}_d(q)$, the symplectic designs, and two exceptional examples, with $\widetilde{v}= 11$ or $176$).
        \item $\widetilde{\D}$ is a  $2$-design with $\widetilde{k}=p+1$ and  $\widetilde{v}=p^2+p+1$ for any positive integer $p\ge 2$, so that $t=1$. For instance, a projective plane or a complete design with those parameters work.  (Note that if $\widetilde{\lambda}=1$ then this is a symmetric design.) 
    \end{enumerate}
\end{remark}

The third remark presents a very general context for application of Construction~\ref{con:wr-comb}, namely the situation where the input design $\widetilde{\D}$ with the input group $\widetilde{G}$ of automorphisms already preserves a poset block structure on $\widetilde{\PP}$.

\begin{remark} \label{rem:poset-con:wr-comb}
For any subgroup $\widetilde{G} \leq \Aut(\widetilde{\D})$, the group $G = \widetilde{G} \wr S_e$ preserves the partition  $\{ \{\widetilde{\PP}\} \times \{\varepsilon\} \ | \ \varepsilon \in \Delta \}$ of $\PP = \widetilde{\PP} \times \Delta$. Suppose that the group $\widetilde{G}$ in Construction \ref{con:wr-comb} preserves a  poset block structure $(\widetilde{\PP}, \widetilde{\C}^*)$ that corresponds to a poset $\widetilde{\I} = (\widetilde{I},\preccurlyeq)$, as described in Subsection \ref{ss:posetblockstruc}. That is,  $\widetilde{\PP} = \Delta_{\widetilde{I}} = \prod_{i \,\in\, \widetilde{I}} \Delta_i$ and  $\widetilde{G}$ is a subgroup of the generalised wreath product $\prod_{(\widetilde{I},\preccurlyeq)} (\Sym(\Delta_i),\Delta_i)$. Then $\widetilde{G}$ preserves the partitions $\widetilde{\C}_{\widetilde{J}}$  of $\widetilde{\PP}$ for all ancestral subsets $\widetilde{J} \subseteq \widetilde{I}$. Each partition $\widetilde{\C}_{\widetilde{J}}$ then gives rise to a partition of the set $\PP$, which we denote by $\widetilde{\C}_{\widetilde{J}} \times \Delta$,  namely
    \[
    \widetilde{\C}_{\widetilde{J}} \times \Delta = \big\{ \widetilde{C} \times \{\varepsilon\} \ \big| \ \widetilde{C} \in \widetilde{\C}_{\widetilde{J}}, \ \varepsilon \in \Delta \big\},
    \]
and the group $G$ preserves each of these partitions. The partition $\widetilde{\C}_{\varnothing} \times \Delta = \{ \widetilde{\PP} \} \times \Delta$ is the `top' partition, that is, each class in any other partition $\widetilde{\C}_{\widetilde{J}} \times \Delta$ is contained in a unique class of $\{ \widetilde{\PP} \} \times \Delta$.  Inside each class  $ \widetilde{\PP} \times \{\epsilon\}$ of $\widetilde{\C}_{\varnothing} \times \Delta$ is a copy of the  poset block structure $(\widetilde{\PP}, \widetilde{\C}^*)$ 
% {\color{lightgray}$\widetilde{\I}$-imprimitivity system} 
corresponding to the poset $\widetilde{\I}$. It follows that $G$ preserves a poset of partitions corresponding to the poset $\I$  obtained by adding a single node `on top of' $\widetilde{\I}$, as illustrated on the left in Figure \ref{fig:poset-nodeontop}, and defined formally as follows:
    \begin{equation} \label{eq:nodeontop}
    \I = (I,\preccurlyeq) \ \ \text{where} \ I = \widetilde{I} \,\dot\cup\, \{s\}, \ \text{the restriction} \ \I|_{\widetilde{I}} = \widetilde{\I}, \ \text{and $i \preccurlyeq s$ for all $i \in \widetilde{I}$.}
    \end{equation}

    \begin{figure}
    \includegraphics[]{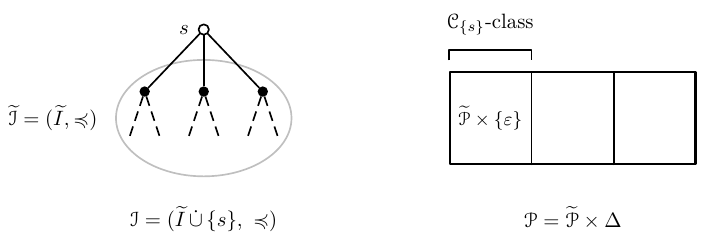}
    \caption{Poset $\I$ and partition $\C_{\{s\}}$ in Remark \ref{rem:poset-con:wr-comb}}
    \label{fig:poset-nodeontop}
    \end{figure}

If we denote the set $\Delta$ by $\Delta_s$ then $\PP = \widetilde{\PP} \times \Delta_s = \prod_{i \in I} \Delta_i$. The ancestral subsets in $\I$ are the sets $\widetilde{J} \cup \{s\}$, where $\widetilde{J}$ is ancestral in $\widetilde{\I}$. For each ancestral subset $\widetilde{J}$ in $\widetilde{\I}$, the partition $\widetilde{\C}_{\widetilde{J}} \times \Delta = \C_{\widetilde{J} \,\cup\, \{s\}}$, and in particular $\widetilde{\C}_{\varnothing} \times \Delta = \{\widetilde{\PP}\} \times \Delta = \C_{\{s\}}$, as illustrated on the right in Figure \ref{fig:poset-nodeontop}.
\end{remark}

The set-up described in Remark~\ref{rem:poset-con:wr-comb} may be regarded as a prototype for the recursive step of a procedure whereby we apply Construction~\ref{con:wr-comb} iteratively. Our next construction formalises such a procedure.

\begin{construction}\label{cons-iterated}
    Let $\widetilde{\D} = (\widetilde{\PP},\widetilde{\B})$ be a $1$-design. %$2$-$(\widetilde{v}, \widetilde{k}, \widetilde{\lambda})$ design.
    Let $e_1, e_2, \ldots, e_n \geq 2$ be integers.
We recursively define designs:
\[\D_0=\widetilde{\D} \quad \text{and \quad for }i=1,2,\ldots, n, \ \ \D_i=\D(\D_{i-1},e_i),\]
where $\D(\D_{i-1},e_i)$ is the design from Construction \ref{con:wr-comb}.
   We also write $\D(\widetilde{\D},e_1, e_2, \ldots, e_n)=\D_n$. 
\end{construction}

It follows from Remark \ref{remark-combcons} that $\D(\widetilde{\D},e_1, e_2, \ldots, e_n)$ is a $1$-design, and in the next theorem we determine precisely when $\D(\widetilde{\D},e_1, e_2, \ldots, e_n)$ is a $2$-design.

\begin{theorem} \label{thm:addchainontop}
Let $\widetilde{\D} = (\widetilde{\PP},\widetilde{\B})$ be a $1$-design and  let $e_1, e_2, \ldots, e_n \geq 2$ be integers. Let $\D_n=\D(\widetilde{\D},e_1, e_2, \ldots, e_n)$ described in Construction \ref{cons-iterated}.

Then $\D_n$ is a $2$-design if and only if 
    \begin{equation} \label{hyp:con-ext}
    \text{\parbox[t]{15cm}{\centering
    $\widetilde{\D}$ is a $2$-design, \ \  $t := \dfrac{\widetilde{k}(\widetilde{k}-1)}{{\widetilde{v}}-1}$ is an integer, \ \ and for $s=1,2,\ldots n,$ \\
    $e_s = \displaystyle\left(\widetilde{k} + \sum_{j=1}^{s-1} (e_j-1) - 1\right)\left(\widetilde{k} + \sum_{j=1}^{s-1} (e_j-1) - 2\right) + t$.}}
    \end{equation}
Moreover, if \eqref{hyp:con-ext} holds and $\widetilde{\D}_n$ is a $2$-$(\widetilde{v}, \widetilde{k}, \widetilde{\lambda})$ design, then $\D$ is a $2$-$(v,k,\lambda)$ design, where 
% \begin{align*}
%     v &= \widetilde{v}\cdot \prod_{j=1}^n e_j\\
%     k &= \widetilde{k}+\sum_{j=1}^{n}(e_j-1)\\
%     \lambda &= \widetilde{\lambda} \cdot \widetilde{v}^{k-\widetilde{k}} \cdot \prod_{1\leq j< i\leq n } e_j^{e_i-1}
% \end{align*}
\begin{equation*}
    v = \widetilde{v}\cdot \prod_{j=1}^n e_j,\quad 
    k = \widetilde{k}+\sum_{j=1}^{n}(e_j-1),\quad 
    \lambda = \widetilde{\lambda} \cdot \widetilde{v}^{k-\widetilde{k}} \cdot \prod_{1\leq j< i\leq n } e_j^{e_i-1}
\end{equation*}
and $k(k-1)/(v-1)=t$.
\end{theorem}

\begin{proof}
We will show by induction that the statement holds for $\D_m$ as in Construction \ref{cons-iterated} for each $m=1,2,\ldots, n.$ We denote by $v_m$ and $k_m$ the number of points and the number of points per block, respectively, of $\D_m$. 

First note that the statement holds for $\D_1$ by Proposition \ref{prop:con-wr-comb} (with the standard conventions for empty sums and empty product).

Assume that the statement holds for $\D_{m-1}$ and consider $\D_m=\D(\D_{m-1},e_i)$. %Denote by 
By Proposition \ref{prop:con-wr-comb}, $\D_m$ is a 2-design if and only if 
$\D_{m-1}$  is a $2$-design, $t_m := \frac{k_{m-1}(k_{m-1}-1)}{v_{m-1}-1}$ is an integer, and $e_m = (k_{m-1}-1)(k_{m-1}-2)+t_m.$
Now by the induction hypothesis, $\D_{m-1}$  is a $2$-design if and only if 
$\widetilde{\D}$ is a $2$-design, $t := \frac{\widetilde{k}(\widetilde{k}-1)}{{\widetilde{v}}-1}$ is an integer and for $s=1,2,\ldots m-1,$ $e_{s} = \big(\widetilde{k}+\sum_{j=1}^{s-1}(e_j-1)-1\big)\big(\widetilde{k}+\sum_{j=1}^{s-1}(e_j-1)-2\big)+t.$
Moreover, in that case $t_{m}=t$ (last induction hypothesis) and $k_{m-1}=\widetilde{k}+\sum_{j=1}^{m-1}(e_j-1)$.
Putting all these conditions together, $\D_m$ is a 2-design if and only if 
$\widetilde{\D}$  is a $2$-design,  $t := \frac{\widetilde{k}(\widetilde{k}-1)}{{\widetilde{v}}-1}$ is an integer and for $s=1,2,\ldots m,$ $e_{s} = \big(\widetilde{k}+\sum_{j=1}^{s-1}(e_j-1)-1\big)\big(\widetilde{k}+\sum_{j=1}^{s-1}(e_j-1)-2\big)+t.$

Now we assume that $\widetilde{\D}$ is a $2$-$(\widetilde{v}, \widetilde{k}, \widetilde{\lambda})$ design. By the induction hypothesis, $\D_{m-1}$ is a $2$-$(v_{m-1},k_{m-1},\lambda_{m-1})$, with parameters as in the statement (for $n=m-1$).
By Proposition \ref{prop:con-wr-comb},
    \[v_m=v_{m-1} e_m= \widetilde{v}\cdot \prod_{j=1}^{m-1} e_j\cdot e_m=\widetilde{v}\cdot \prod_{j=1}^{m} e_j\]
and 
    \[k_m=k_{m-1}+e_m-1=\widetilde{k}+\sum_{j=1}^{m-1}(e_j-1)+e_m-1=\widetilde{k}+\sum_{j=1}^{m}(e_j-1).\]
Finally
    \begin{align*}
    \lambda_m
    &=\lambda_{m-1}\cdot v_{m-1}^{e_m-1}\\
    &=\widetilde{\lambda} \cdot \widetilde{v}^{k_{m-1}-\widetilde{k}} \cdot \prod_{1\leq j< i\leq m-1 } e_j^{e_i-1}\cdot \left(\widetilde{v}\cdot \prod_{j=1}^{m-1} e_j\right)^{e_m-1}\\
    &=\widetilde{\lambda} \cdot \widetilde{v}^{k_{m-1}+e_m-1-\widetilde{k}} \cdot \prod_{1\leq j< i\leq m-1 } e_j^{e_i-1}\cdot \left(\prod_{j=1}^{m-1} e_j\right)^{e_m-1}\\
    &=\widetilde{\lambda} \cdot \widetilde{v}^{k_m-\widetilde{k}} \cdot \prod_{1\leq j< i\leq m } e_j^{e_i-1}.
    \end{align*}
This concludes the proof.
\end{proof}

We now make a series of remarks and examples which show how to apply Construction~\ref{cons-iterated} to produce explicit families of $2$-designs. The first remark formally shows how Remark~\ref{rem:poset-con:wr-comb} can be applied recursively.

\begin{remark}\label{rem:addchainontop}
The design $\D(\widetilde{\D}, e_1, e_2, \ldots, e_n)$ in Construction \ref{cons-iterated} has point set $\PP = \widetilde{\PP} \times \Delta_1 \times \Delta_2 \times \cdots \times \Delta_n$, where $|\Delta_i| = e_i$ for each $i \in \{1, 2, \ldots, n\}$. Applying Remark \ref{rem:poset-con:wr-comb} iteratively we see that, for any subgroup $\widetilde{G} \leq \Aut(\widetilde{\D})$ that preserves a poset block structure $(\widetilde{\PP}, \widetilde{\C}^*)$ corresponding to a poset $\widetilde{\I} = (\widetilde{I},\preccurlyeq)$ as described in Subsection \ref{ss:posetblockstruc}, the iterated wreath product $G = \widetilde{G} \wr S_{e_1} \wr \ldots \wr S_{e_n}$ preserves a poset block structure on $\PP$ corresponding to a poset $\I$ defined as follows:

    \begin{equation} \label{e:chain-on-top}
    \I = (I,\preccurlyeq) \ \ \text{where} \ \ \text{\parbox[t]{12.5cm}{\raggedright $I = \widetilde{I} \,\dot\cup\, H$ for a chain $\mathscr{H} = (H,\preccurlyeq)$ with $n$ nodes; \ \ the restrictions \ $\I|_{\widetilde{I}} = \widetilde{\I}$ \ and \ $\I|_{H} = \H$;  \  \text{and} \ $j \preccurlyeq i$ \ \  $\forall\,j \in \widetilde{I}, \ i \in H$.}}
    \end{equation}
    
That is, $\I$ is the poset obtained by adding a chain $\mathscr{H}$ `on top of' $\I$, as shown in Figure \ref{fig:addchainontop}.
    \begin{figure}
    \includegraphics[height=5cm]{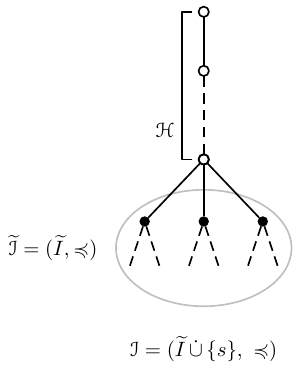}
    \caption{Poset $\I$ in Remark \ref{rem:addchainontop}}
    \label{fig:addchainontop}
    \end{figure}
\end{remark}

The next remark is a technical observation about the parameters in Theorem~\ref{thm:addchainontop}.

\begin{remark} \label{rem:firsttwoes}
Computing the first two values $e_1, e_2$ in Theorem \ref{thm:addchainontop} yields
\begin{align*}
    e_{1} &= (\widetilde{k}-1)(\widetilde{k}-2) + t=\widetilde{k}^2-3\widetilde{k}+t+2 \\
    e_{2} &=\left(\widetilde{k} + (e_1-1) - 1\right)\left(\widetilde{k} + (e_1-1) - 2\right) + t\\
   &=\left( \widetilde{k}^2-2\widetilde{k}+t \right)\left( \widetilde{k}^2-2\widetilde{k}+t-1\right) + t\\
    &= ((\widetilde{k}-1)^2+t-1)((\widetilde{k}-1)^2+t-2)+t \\
    &= (\widetilde{k}-1)^4 + (\widetilde{k}-1)^2 (2t-3) + (t-1)(t-2) + t \\
    &=(\widetilde{k}-1)^2((\widetilde{k}-1)^2+2t-3)+t^2-2t+2\\
    &=(\widetilde{k}-1)^2(\widetilde{k}^2-2\widetilde{k}+2t-2)+t^2-2t+2.
\end{align*}
% {\color{red}  are we still using this observation? {\color{magenta} CEP: This remark only referred to at end of proof of Prop~\ref{prop:con-rect-comb}, so think it's safe to remove the following.} I think we don't, so we should delete it.
% Observe that
%     \begin{align*}
%     (e_{1}-1)(\widetilde{k}(\widetilde{k}-1) + t - 1)
%     &= ((\widetilde{k}-1)(\widetilde{k}-2) + t - 1)(\widetilde{k}(\widetilde{k}-1) + t - 1) \\
%     &= (\widetilde{k}-1)^2 \widetilde{k}(\widetilde{k}-2) + (\widetilde{k}-1)(t-1)(2\widetilde{k}-2) + (t-1)^2 \\
%     &= (\widetilde{k}-1)^2 (\widetilde{k}(\widetilde{k}-2) + 2t - 2) + (t-1)^2 \\
%     &= e_{2}-1 
%     \end{align*}}
\end{remark}

For many of the known examples the parameters $e_i$ can be expressed quite compactly and uniformly and the parameter  
$t = 1$. We give one family of such parameters in Lemma~\ref{lem:chainontop-p+1}. 

\begin{lemma} \label{lem:chainontop-p+1}
% {\color{lightgray} Under the hypotheses of Theorem \ref{thm:addchainontop}, suppose that the input $1$-design $\widetilde{\D}$ has parameters $\widetilde{k} = p+1$ and $\widetilde{v} = p^2 + p + 1$, for some positive integer $p$, and let $H = \{1,2,\ldots n\}$ form a chain $1 \preccurlyeq 2 \preccurlyeq \ldots \preccurlyeq n$ with $n$ nodes.}
Under the hypotheses of Theorem \ref{thm:addchainontop}, suppose that the input  $2$-design $\widetilde{\D}$ has parameters $\widetilde{k} = p+1$ and $\widetilde{v} = p^2 + p + 1$, for some positive integer $p$, and let $\mathscr{H}=(H, \preccurlyeq)$ be a chain with $n$ nodes. 
Then the parameters of the $2$-design $\D_n=\D(\widetilde{\D},e_1, e_2, \ldots, e_n)$  arising from Construction~\ref{cons-iterated} and Theorem \ref{thm:addchainontop} are 
    \begin{align*}
    e_{i} &= p^{2^i} - p^{2^{i-1}} + 1 \text{ for all } i = 1,2,\ldots,n \\
    k &= p^{2^n} + 1 \\
    v &= p^{2^{n+1}} + p^{2^n} + 1.
    \end{align*}
\end{lemma}

\begin{proof}
First we compute $t = \frac{\widetilde{k}(\widetilde{k}-1)}{\widetilde{v}-1} = \frac{(p+1)p}{p^2+p} = 1$, and note that it is an integer. The rest of our proof is by induction on $n$. For the case $n=1$, we use Construction \ref{con:wr-comb} and Proposition~\ref{prop:con-wr-comb} to obtain 
    \begin{align*}
    e_{1} &= (\widetilde{k}-1)(\widetilde{k}-2) + t
    = p(p-1) + 1 = p^2-p+1 \\
    k &= (\widetilde{k}-1)^2 + t = p^2+1 \\
    v &= \widetilde{v} e_{1}
    = (p^2+p+1)(p^2-p+1)
    =(p^2+1)^2 - p^2
    = p^4+p^2+1
    \end{align*}
as asserted. 

Now assume that $n \geq 2$ and that the parameters are as stated for chains with $n-1$ nodes, and consider a chain with $n$ nodes.
Then by the induction hypothesis (using the same notation as in the proof of Construction \ref{cons-iterated} and Theorem \ref{thm:addchainontop}), for the design $\D_{n-1}$ corresponding to the chain with $n-1$ nodes, we have
    \begin{align*}
    e_{i} &= p^{2^i} - p^{2^{i-1}} + 1 \text{ for all } i \leq n-1 \\
    k_{n-1} &= p^{2^{n-1}} + 1 \\
    v_{n-1} &= p^{2^{n}} + p^{2^{n-1}}+1 
    \end{align*}
Moreover, by Theorem \ref{thm:addchainontop}, $\frac{k_{n-1}(k_{n-1}-1)}{v_{n-1}-1} = t = 1$,
and applying Construction \ref{con:wr-comb} to the design $\D_{n-1}$, we have, for the design $\D_n$ (corresponding to the chain of with $n$ nodes), 
    \begin{align*}
    e_{n} &= (k_{n-1}-1)(k_{n-1}-2) + t
    = p^{2^{n-1}}(p^{2^{n-1}} - 1)+1
    = p^{2^n} - p^{2^{n-1}} + 1 \\
    k &= (k_{n-1}-1)^2 + t
    = p^{2^{n}} + 1 \\
    v &=v_{n-1} e_{n}
    = (p^{2^n} + p^{2^{n-1}} + 1)(p^{2^n} - p^{2^{n-1}} + 1)
    = p^{2^{n+1}} + p^{2^{n}} + 1
    \end{align*}
Thus the result follows by induction.
\end{proof}

One of the first explicit infinite families of `extremely imprimitive' block-transitive $2$-designs we constructed was in \cite[Construction 4.4]{chainspaper}, where the group preserved a poset block structure corresponding to a chain of arbitrary length. Example~\ref{ex:chains} shows how the approach in this paper could have been used for that construction.

\begin{example}[Designs corresponding to chains] \label{ex:chains}
% {\color{lightgray}
% The designs $\D^s$, for integers $s \geq 2$, obtained in  \cite[Construction 4.4]{chainspaper} can also be obtained by applying Construction \ref{cons-iterated} in the manner described in Theorem \ref{thm:addchainontop}. The corresponding partially ordered set $(I,\preccurlyeq)$ is a chain where $I = \{1, \ldots, s\}$ with $i \preccurlyeq i+1$ for $1 \leq i \leq s-1$. Thus
%     \[ I = \widetilde{I} \,\dot\cup\, H \quad \text{where } \widetilde{I} = \{1\}, \ H = \{2, \ldots, s\} \]
% and the subsets $\widetilde{I}$ and $H$ satisfy condition \eqref{e:chain-on-top}. The ``base'' design $\widetilde{\D}$ is any design with point set $\Delta_1$ of size $p^2 + p + 1$ and block set $B^{G_1}$, where $B$ is any point subset of size $p+1$ and $G_1$ is a $2$-transitive subgroup of $\Sym(\Delta_1)$. (There are several possibilities for the group $G_1$ and design $\widetilde{\D}$. For instance, we could take $G_1 = S_{p^2 + p + 1}$ with $\widetilde{\D}$ the complete design; or if $p$ is a prime power then we could take $G_1 = {\rm PGL}(3,p)$ with $\widetilde{\D}$ the Desarguesian projective plane of order $p$.) For each $i \in \{2, \ldots, s\}$ the parameter $e_i= p^{2^{i-1}} - p^{2^{i-2}} + 1$, as in Lemma \ref{lem:chainontop-p+1}.}
Suppose that $\widetilde{\D}$ is a $2$-($\widetilde{v},\widetilde{k},\widetilde{\lambda}$) design such that $t = \widetilde{k}(\widetilde{k}-1)/(\widetilde{v}-1)$ is an integer, and with a block-transitive, point-primitive automorphism group $\widetilde{G}$. Then the poset block structure preserved by $\widetilde{G}$ is isomorphic to the poset $\widetilde{\I}$ consisting of a single node. For integers $e_1, e_2, \ldots, e_n$ satisfying condition \eqref{hyp:con-ext}, it follows from Theorem~\ref{thm:addchainontop} that
the design $\D = \D(\widetilde{\D},e_1,e_2,\ldots,e_n)$ obtained by applying Construction \ref{cons-iterated} is a block-transitive $2$-design with poset block structure isomorphic to a chain $\I$ with $n+1$ nodes and satisfying condition \eqref{e:chain-on-top}.

There are several possibilities for the input design $\widetilde{\D}$, for instance, any of the designs listed in Remark \ref{r:con-wr-ex} (a)--(d). In particular, suppose that $\widetilde{\D}$ is as described in Remark \ref{r:con-wr-ex} (d), that is, $\widetilde{\D}$ has point set $\widetilde{\PP}$ of size $p^2 + p + 1$ and blocks of size $p+1$, for any integer $p \geq 2$.  Suppose that $\widetilde{G}$ is a $2$-transitive subgroup of $\Sym(\widetilde{\PP})$ and the block set $\widetilde{\B} = \widetilde{B}^{\widetilde{G}}$ for a block $\widetilde{B}$. (For instance, we could take $\widetilde{G} = S_{p^2 + p + 1}$, so that $\widetilde{\D}$ is the complete design, while if $p$ is a prime power then we could alternatively take $\widetilde{G} = {\rm PGL}(3,p)$ with $\widetilde{\D}$ the Desarguesian projective plane of order $p$.) Then by Lemma \ref{lem:chainontop-p+1} the parameters $e_i = p^{2^i} - p^{2^{i-1}} + 1$ for each $i \in \{1, 2, \ldots, n\}$, and if $\widetilde{G} = S_{p^2 + p + 1}$ then $\D$ is the design denoted $\D^{n+1}$ in \cite[Construction 4.4]{chainspaper}. (Note that we use a slightly different notation, since in \cite{chainspaper} the number $p^{2^i} - p^{2^{i-1}} + 1$ is denoted by $e_{i+1}$.)
\end{example}

We next exploit Construction~\ref{con:wr-comb} to construct $2$-designs preserving a poset block structure for an inverted V-shaped poset.

\begin{example}[Designs corresponding to the three node inverted V-shaped poset]  \label{ex:V-inv}

\noindent
Here the partially ordered set $\I = (I,\preccurlyeq)$ is the inverted V-shaped poset where $I = \{1,2,3\}$, such that  $1 \prec 3$, $2 \prec 3$, $1 \not\prec 2$, and $2 \not\prec 1$, as shown in Figure \ref{fig:invVposet}.
    \begin{figure}
    \includegraphics[height=4cm]{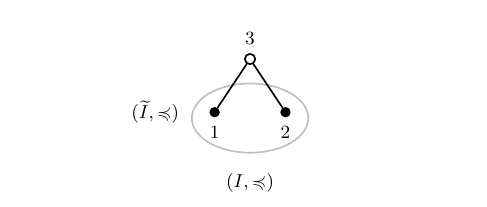}
    \caption{Poset $\I$ in Example \ref{ex:V-inv}}
    \label{fig:invVposet}
    \end{figure}
We write $I = \widetilde{I} \,\dot\cup\, H$ where $\widetilde{I} = \{1,2\}$ and $H=\{3\}$, and note that $\widetilde{I}$ and $H$ satisfy condition \eqref{eq:nodeontop}, and that $(\widetilde{I},\preccurlyeq)$ is a $2$-antichain.
By Proposition~\ref{prop:con-wr-comb}, using Construction \ref{con:wr-comb},  a $2$-design $\D$ admitting an inverted V-shaped poset block structure can be obtained from an input grid-imprimitive $2$-design $\widetilde{\D}$ (that is, there are two nontrivial invariant point-partitions which can be visualised as the set of rows and the set of columns of a rectangular grid, corresponding to the $2$-antichain) provided the parameters $\widetilde{k}$ and $\widetilde{v}$ of $\widetilde{\D}$ are such that $\widetilde{k}(\widetilde{k}-1)/(\widetilde{v}-1)$ is an integer.

One explicit family of $2$-designs admitting an inverted V-shaped poset block structure was previously constructed in \cite[Example 5.3]{SmallEx}, and we describe here how this family of $2$-designs could be obtained using Construction~\ref{con:wr-comb}.  Take $\widetilde{\PP} = \Delta_1 \times \Delta_2$ with $|\Delta_i|=p^2-(-1)^i p + 1$, for each $i$, for an arbitrary positive integer $p$, so that $|\widetilde{\PP} |=\widetilde{v}= p^4+p^2+1$. Take $\widetilde{B}$ to be  the subset of size $\widetilde{k}=p^2+1$ of the rectangular grid shown in the left-hand rectangle of \cite[Figure 6]{SmallEx}. Take $\widetilde{G}=\widetilde{G_1}\times \widetilde{G_2}$ for $2$-transitive subgroups $\widetilde{G_i}\leq \Sym(\Delta_i)$, and let $\widetilde{\B}=\widetilde{B}^{\widetilde{G}}$. Then $\widetilde{\D} = (\widetilde{\PP}, \widetilde{\B})$ is a grid-imprimitive 2-design satisfying $t = \widetilde{k}(\widetilde{k} - 1)/(\widetilde{v} - 1) = 1$, and hence, by Proposition~\ref{prop:con-wr-comb}, we obtain a $2$-design $\D(\widetilde{\D},e)$ from Construction~\ref{con:wr-comb} if we take the parameter  $e=p^2(p^2-1)+1 = p^4-p^2+1$. This gives precisely the $2$-designs described in \cite[Example 5.3]{SmallEx}  (where our parameter $e$ was denoted $e_3$); this can be seen by comparing \cite[Figure 6]{SmallEx} and Figure \ref{fig:con-wr-comb}. 

Other grid-imprimitive 2-designs can be used as base designs in our construction. In \cite[Construction 7.4, Figure 3]{multigrids}, a grid-imprimitive $2$-design  was constructed with the same parameters as in the previous paragraph, but with different base block $\widetilde{B'}$, and with the group $\widetilde{G}=S_{e_1}\times S_{e_2}$. 
We can replace $\widetilde{G}$ by $\widetilde{G}_1\times \widetilde{G}_2$, for any $2$-transitive subgroups $\widetilde{G}_i\leq S_{e_i}$, and still obtain a $2$-design $\widetilde{\D'}$, taking as blocks the images of $\widetilde{B'}$ under elements of $\widetilde{G}_i\leq S_{e_i}$. (This follows from \cite[Theorem 1.2 and Lemma 4.4]{SmallEx}.) 
Note that for $p=2$, the two base blocks $\widetilde{B}$ and $\widetilde{B'}$ are the same, but that is not the case for $p>2$. It is not obvious however if the $2$-designs from the previous paragraph and from \cite[Construction 7.4]{multigrids} are non-isomorphic.  This is hard to check computationally as the number of blocks is very large. We were able to check  using Magma that when $p=3$ and $\widetilde{G}= \AGL(1,13)\times \AGL(1,7)$, the two  designs are indeed non-isomorphic (and hence the designs obtained after applying Construction \ref{con:wr-comb} to these two designs are also non-isomorphic), but we do not know if this is true in general for other values of $p$ and choices of $\widetilde{G}=\widetilde{G_1}\times \widetilde{G_2}$.

%{\color{lightgray}Note that the input $2$-design $\widetilde{\D} = (\widetilde{\PP},\widetilde{\B})$ used in the previous paragraph can depend on the group $G$ used to generate the block-set $\widetilde{\B}=B^G$. In  \cite[Example 5.3]{SmallEx}, $G$ was taken as $G_1\times G_2$ for $2$-transitive subgroups $G_i\leq \Sym(\Delta_i)$. We found computationlly using Magma, for example, that when $p=3$, the $2$-design $\widetilde{\D}$ we obtain if we take $(G_1,G_2) = (\PSL(3,3),\PSL(3,2))$ is non-isomorphic to the $2$-design we get with $(G_1,G_2) = (\AGL(1,13),\AGL(1,7))$. }

%Variants of the above construction can also yield new examples,\; for instance, some of the $2$-designs  obtained from \cite[Construction 7.4]{multigrids} (which have the same parameters as the designs in \cite[Example 5.3]{SmallEx}).  
Also we list a few explicit examples of other possible designs $\widetilde{\D}$ mentioned  \cite{grids22}, which admit block-transitive grid-imprimitive groups and which we could have used, and which give different designs $\D(\widetilde{\D},e)$ from those above; namely those listed in the table below which were constructed in \cite{grids22}. The last column lists $e$ as in \eqref{hyp:con} so that $\D(\widetilde{\D},e)$  is a $2$-design.  We note that \cite{grids22} also lists designs for which the parameter $t$ is not an integer, so the integrality condition \eqref{hyp:con} in our construction can be a genuine restriction on the designs $\widetilde{\D}$.

    \begin{center}
    \begin{tabular}{ll|rrrr}
  $u$-$(\widetilde{v},\widetilde{k},\widetilde{\lambda}) $ design $\widetilde{\D}$  %$\widetilde{\D}$ 
   & Reference for $\widetilde{\D}$ & $|\Delta_1|$ & $|\Delta_2|$ %& $\widetilde{v}$ & $\widetilde{k}$ 
   & $t$ & $e$ \\
    \hline\hline
   $2$-$(4,3,2)$ (complete design) &
    \cite[Example 1]{grids22} with $m = 2$ & $2$ & $2$ %& $4$ & $3$ 
    & $2$ & $4$ \\
    $2$-$(4,4,1)$ (trivial design)
    &\cite[Example 2]{grids22} with $m = 2$ & $2$ & $2$ %& $4$ & $4$ 
    & $4$ & $10$ \\
    $2$-$(16,6,12)$ %(not a $3$-design)
    &\cite[Example 2]{grids22} with $m = 4$ & $4$ & $4$ %& $16$ & $6$ 
    & $2$ & $22$ \\
    $3$-$(16,6,80)$, also  $2$-$(16,6,280)$ 
    &\cite[Figure 2]{grids22} & $2$ & $8$ %& $16$ & $6$ 
    & $2$ & $22$
    \end{tabular}
    \end{center}
% {\color{blue} %do we need to explain $e_1,e_2,e_3$?
% First two rows are complete and trivial designs on 4 points, so they are not genuinely grid-imprimitive. We need to check if rows 3,4 are or not, and also are they different designs.
% }
\end{example}

All the designs in Example~\ref{ex:V-inv} can be used as input for the recursive process given in Construction~\ref{con:rect-comb} to obtain $2$-designs preserving an inverted Y-shaped poset block structure.

\begin{example}[Designs corresponding to an inverted-Y shaped poset] \label{ex:invertedY}
Consider an inverted-Y shaped poset $\I =  (I, \preccurlyeq)$ with $I = \{1, 2, \ldots, s\}$ (for some $s \geq 4$), and with $1 \prec i$ and $2 \prec i$ for all $i \geq 3$,  $1 \not\prec 2$ and $2 \not\prec 1$, as shown in Figure~\ref{fig:invertedY} (the black dots indicate the nodes for $\widetilde{\I}$).
Then $I = \widetilde{I} \,\dot\cup\, H$, where $\widetilde{I} = \{1,2\}$ and $H = \{3, \ldots, s\}$, $(\widetilde{I},\preccurlyeq)$ is a $2$-antichain, $(H, \preccurlyeq)$ is a chain of length $s-2$, and $\widetilde{I}$ and $H$ satisfy condition \eqref{e:chain-on-top}. Then each of the  $2$-designs $\widetilde{\D}$ in Example \ref{ex:V-inv} admitting a block-transitive grid-imprimitive group of automorphisms can be used in Construction \ref{cons-iterated} to construct a $2$-design admitting a block-transitive group preserving an $\I$-poset block structure. 
\end{example}

\begin{figure}
    \centering
    \includegraphics[width=0.5\linewidth]{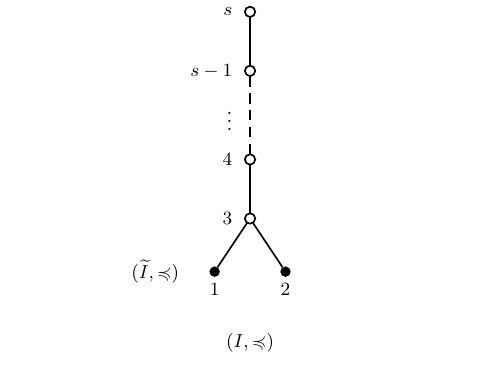}
    \caption{Poset for Example~\ref{ex:invertedY}}
    \label{fig:invertedY}
\end{figure}

Finally  we mention several additional applications of Construction~\ref{con:rect-comb} which produce $2$-designs admitting block-transitive groups that preserve an $\I$-poset block structure on points, for one of the posets $\I$ in Figure~\ref{fig:smallposets}. 

\begin{example}[Designs corresponding to other posets] \label{ex:smallposets}
The $2$-designs constructed in \cite[Examples 5.1, 5.2, and 5.4]{SmallEx} and \cite[Construction 7.7]{multigrids} all satisfy condition \eqref{hyp:con} with $t = 1$. Take $\widetilde{\D}$ to be one of these designs. Then applying Construction \ref{cons-iterated}  produces a $2$-design admitting a block-transitive group preserving an $\I$-poset block structure, where $\I$ is one of the posets pictured in Figure \ref{fig:smallposets}, with the nodes of the poset $\widetilde{\I}$ indicated by the black dots.

\begin{figure}
    \centering
    \includegraphics[width=0.7\linewidth]{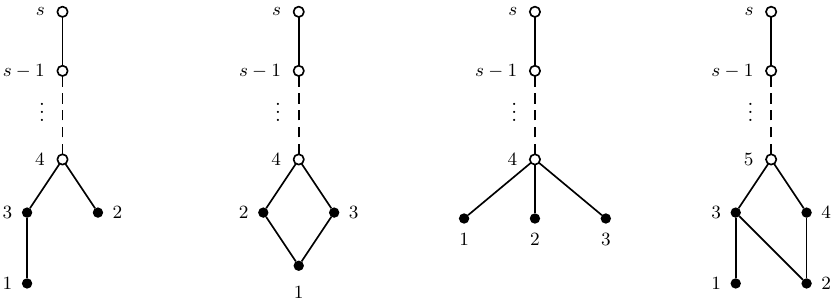}
    \caption{Posets for Example~\ref{ex:smallposets}}
    \label{fig:smallposets}
\end{figure}
\end{example}

\section{Construction: adding two independent nodes on top of the poset}

For our second general construction, we again begin with a basic step (Construction~\ref{con:rect-comb}) which takes as input a $1$-design and, this time, a pair of positive integers. As with Construction~\ref{con:wr-comb} it can produce $2$-designs under appropriate conditions on the parameters (Proposition~\ref{prop:con-rect-comb}). We explain in Remark~\ref{rem:twoontop-poset} how, given an input $2$-design with a block-transitive group preserving a poset block structure, Construction~\ref{con:rect-comb}  produces $2$-designs with block-transitive groups preserving poset block structures for posets involving two additional nodes. We then apply this theory in Example~\ref{ex:V}, and in Example~\ref{ex:Y}, to construct $2$-designs admitting a block-transitive group preserving a poset block structure for the V-shaped poset with three nodes, and for Y-shaped posets with the  stem of the Y arbitrarily large.  Finally in Example~\ref{ex:smallposets-2nodesontop} we give constructions of $2$-designs using Construction~\ref{con:rect-comb} for several other small posets.

\begin{construction} \label{con:rect-comb} 
Let  $\widetilde{\D} = (\widetilde{\PP},\widetilde{\B})$ be a $1$-design where $\widetilde{v} := |\widetilde{\PP}|$ and each block has size $\widetilde{k}\geq 2$. The number of blocks of $\widetilde{\D}$ is denoted by $\widetilde{b}$, and the number of blocks through a point is denoted by $\widetilde{r}$.
 Let $e_1,e_2>1$ be integers such that $e_1-1$ divides $e_2-1$ (so in particular $e_2\geq e_1)$, let $\Delta_i:=\{1,2,\ldots, e_i\}$ for $i=1,2$ and $\Delta=\Delta_1\times\Delta_2$. Let $m:=\frac{e_2-1}{e_1-1}$. Define 
    \[ \PP = \widetilde{\PP}\times \Delta. \]
We define $\B = \bigcup_{(i,j)\in\Delta} \B_{(i,j)}$ as follows: For each  $(i,j) \in \Delta$, a subset $B \subset \PP$ is in $\B_{(i,j)}$ if and only if there exists an $m$-to-$1$ function $\Phi:(\Delta_2\setminus\{j\}) \to (\Delta_1\setminus\{i\})$ such that
      \[
    \left\{ \begin{array}{lll}
    B \cap (\widetilde{\PP} \times \{(i,j)\}) &= &\widetilde{B} \times \{(i,j)\} \text{ for some } \widetilde{B} \in \widetilde{\B},  \\
    |B \cap (\widetilde{\PP} \times \{(x,j)\})| &= &1 \text{ if $x \neq i$}, \\
    % |B \cap (\widetilde{\PP} \times \{(a,j)\})| = 0 \text{ if $a \neq i$}, \\
        |B \cap (\widetilde{\PP} \times \{(x,y)\})| &= &1 \text{ if $x \neq i,y\neq j$ and $y\Phi=x$}, \ \text{ and}\\
        B \cap (\widetilde{\PP} \times \{(x,y)\})&= &\varnothing \text{ otherwise}
    %      |B \cap (\widetilde{\PP} \times \{(a,b)\})| = 0 \text{ if $a \neq i,b\neq j$ and $b\Phi\neq a$}, \\
    \end{array}\right.
    \]
Let $\D(\widetilde{\D},(e_1,e_2)) := (\PP,\B)$.

\begin{figure}
\begin{center}
    \includegraphics[width=0.8\textwidth]{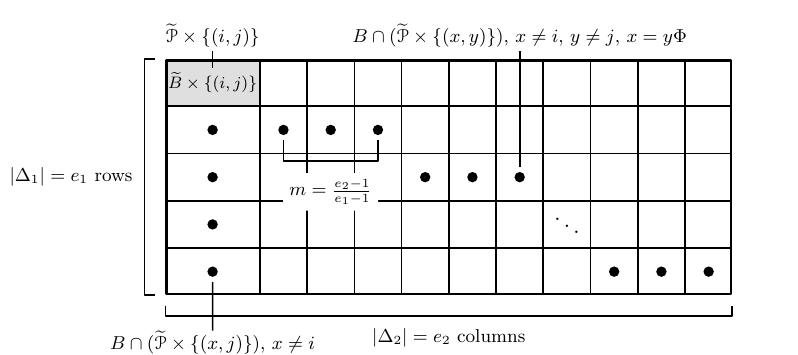}
\end{center}
\caption{Block $B \in \B_{(i,j)}$ in Construction \ref{con:rect-comb}}\label{fig:blockset-rect}
\end{figure}
  \end{construction}

 \begin{remark}\label{remark-rectcombcons}
In Construction \ref{con:rect-comb}, we see that the set $\PP$ has size $v = \widetilde{v} \cdot e_1\cdot e_2$, and each $B \in \B$ has size $k := |B| = \widetilde{k} + e_1+e_2-2 $. %, and the total number of blocks is ?? $|\B| = e \cdot \widetilde{b} \cdot \widetilde{v}^{e-1}$.
We note that the design $\D(\widetilde{\D},(e_1,e_2))$ is always a $1$-design, since it is straightforward to compute  the number $r$ of blocks containing a point of $ \widetilde{\PP} \times \Delta$ in terms of the number $\rho_{e_1,e_2}$ of  $m$-to-$1$ functions $\Delta_2\setminus\{j\}\to \Delta_1\setminus\{i\}$ as follows: 
\begin{align*}
    r
   & = \widetilde{r} \cdot \rho_{e_1,e_2}\cdot \widetilde{v}^{e_1+e_2-2} + (e_1-1) \cdot \widetilde{b} \cdot \widetilde{v}^{e_1-2}\cdot \rho_{e_1,e_2}\cdot \widetilde{v}^{e_2-1} \\
   &\phantom{=}\;\; +(e_1-1)(e_2-1)\cdot \widetilde{b} \cdot \widetilde{v}^{e_1-1}\cdot \frac{\rho_{e_1,e_2}}{e_1-1}\cdot \widetilde{v}^{e_2-2}\\
   & = \rho_{e_1,e_2}\cdot \widetilde{v}^{e_1+e_2-3}(\widetilde{r}\cdot \widetilde{v}+ (e_1+e_2-2) \cdot \widetilde{b})\\
     & = \rho_{e_1,e_2}\cdot \widetilde{v}^{e_1+e_2-3}\cdot \widetilde{b}\cdot (\widetilde{k}+ e_1+e_2-2) \\
        & = \rho_{e_1,e_2}\cdot \widetilde{v}^{e_1+e_2-3}\cdot \widetilde{b}\cdot k 
    \end{align*}
where for the second last equality we use \eqref{eq:vr=bk}.
It is easy to see that $\Aut(\widetilde{\D}) \wr (S_{e_1}\times S_{e_2}) \leqslant \Aut(\D(\widetilde{\D},(e_1,e_2)))$. Moreover, if $\widetilde{G} \leqslant \Aut(\widetilde{\D})$ is transitive on $\widetilde{\B}$, then $\widetilde{G} \wr (S_{e_1}\times S_{e_2})$ is transitive on $\B$. 
\end{remark}

The next remark provides an expression for  $\rho_{e_1,e_2}$ and also several related quantities.

\begin{remark}\label{rem:countingfunctions}
We see as follows that, for positive integers $e_1, n$ with $e_1\geq2$, the number of $m$-to-$1$ functions from a set $\Delta_2'$ of size $e_2-1=m(e_1-1)$ to a set $\Delta_1'$ of size $e_1-1$ is 
    \begin{equation}\label{eq:rho}
    \rho_{e_1,e_2}=\frac{(e_2-1)!}{(m!)^{e_1-1}}.
    \end{equation}
Indeed such a function corresponds to an ordered partition $(\delta_i)_{i\in\Delta_1'}$ of $\Delta_2'$, with $\delta_i$ being the set of elements in $\Delta_2'$ that map to $i$, so $\delta_i$ has size $n$ for each $i$. Different ordered partitions correspond to different $n$-to-1 functions. 
Each ordering of the domain yields such a partition:  the first $n$ elements form $\delta_{i_1}$, the second set of $n$ elements  form $\delta_{i_2}$, and so on. But there are multiple orderings that yield the same ordered partition, namely if we apply to each of these  $e_1-1$ subsets of size $m$ an arbitrary permutation from $S_m$. Thus there are ${(m!)^{e_1-1}}$ orderings that yield the same ordered partition, and hence function.
%Further, if we apply to each of these  $e_1-1$ subsets of size $n$ an arbitrary permutation from $S_n$, then we obtain the same $n$-to-1 function. Thus $\Phi$ depends only on the ordered partition $(\delta_1,\dots,\delta_{e_1-1})$ of $\Delta_2'$ with $e_1-1$ subsets of size $n$, and Now the symmetric group $S_{e_2-1}$ acts transitively on the set of \emph{unordered} partitions of $\Delta_2'$ into $e_1-1$ subsets each of size $n$, with stabiliser $S_{n}\wr S_{e_1-1}$, so the number of unordered such partitions is {\color{red} I find this explanation quite confusing, wasn't it clear enough before?}
%\[
%|S_{e_2-1} :S_{n}\wr S_{e_1-1}| = \frac{(e_2-1)!}{n!^{e_1-1}\cdot (e_1-1)!}
%\]
%(This is true even if $e_1=2$.) Also each unordered 
%partition of this kind corresponds to exactly $(e_1-1)!$ ordered partitions. 
Thus the number of ordered partitions, and therefore the quantity $\rho_{e_1,e_2}$ is as in \eqref{eq:rho}.

We will need further information in our analysis: if the image of one given point $\alpha\in \Delta_2'$ is fixed and equal to $\gamma$, then the corresponding number of $m$-to-$1$ functions is equal to  the number of  ordered partitions $(\delta_i)_{i\in\Delta_1'}$ of $\Delta_2'\setminus\{\alpha\}$ with $|\delta_\gamma|=m-1$ and all other $\delta_i$ of size $m$,  and a similar argument shows that this is 
\begin{equation}\label{eq:nto1fctfixingone}
        \frac{(e_2-2)!}{(m-1)!(m!)^{e_1-2}}=\rho_{e_1,e_2} \frac{m}{e_2-1}=\frac{\rho_{e_1,e_2}}{e_1-1}.
    \end{equation}
Similarly, if $m\geq2$ and 
two given distinct points $\alpha, \beta\in \Delta_2'$ have the same fixed image $\gamma$, then the corresponding number of $n$-to-1 functions is equal to  the number of  ordered partitions $(\delta_i)_{i\in\Delta_1'}$ of $\Delta_2'\setminus\{\alpha,\beta\}$ with $|\delta_\gamma|=m-2$ ($\delta_\gamma$ is empty if $m=2$) and all other $\delta_i$ of size $m$,  and a similar argument shows that this is     \begin{equation}\label{eq:nto1fctfixingtwosameimage}
    \frac{(e_2-3)!}{(m-2)!(m!)^{e_1-2}}=\rho_{e_1,e_2} \frac{m(m-1)}{(e_2-1)(e_2-2)}=\rho_{e_1,e_2} \frac{(m-1)}{(e_1-1)(e_2-2)}.
    \end{equation}
Note there are no such functions if $n=1$, that is, if $e_1=e_2$, and provided $e_2\geq3$ this agrees with  \eqref{eq:nto1fctfixingtwosameimage}.

Finally, if $e_1\geq3$ (so also $e_2\geq 3$  since $e_2\geq e_1$) and 
two given distinct points $\alpha, \beta\in \Delta_2'$ have distinct fixed images $\gamma,\epsilon$ respectively, then the corresponding number of $m$-to-$1$ functions is equal to  the number of  ordered partitions $(\delta_i)_{i\in\Delta_1'}$ of $\Delta_2'\setminus\{\alpha,\beta\}$ with $|\delta_\gamma|=|\delta_\epsilon|=m-1$ and all other $\delta_i$ of size $m$,  and  this is
\begin{equation}\label{eq:nto1fctfixingtwodistinctimage}
    \frac{(e_2-3)!}{(m-1)!^2m!^{e_1-3}}=\rho_{e_1,e_2} \frac{m^2}{(e_2-1)(e_2-2)}=\rho_{e_1,e_2} \frac{m}{(e_1-1)(e_2-2)}.
    \end{equation}
Note there are no such functions if $e_2=2$ (and the formula \eqref{eq:nto1fctfixingtwodistinctimage} is not valid in this case).

\end{remark}

\begin{proposition} \label{prop:con-rect-comb}
The point-block structure $\D(\widetilde{\D},(e_1,e_2))$ in Construction \ref{con:rect-comb} is a $2$-design if and only if 
    \begin{equation}\label{hyp:con-rect}
        \text{\parbox[t]{0.8\textwidth}{\centering
        $\widetilde{\D}$ is a $2$-design, $t := \frac{\widetilde{k}(\widetilde{k}-1)}{{\widetilde{v}}-1}$ is an integer,  \\ $e_1 = ({\widetilde{k}}-1)({\widetilde{k}}-2) + t$, and $e_2=(\widetilde{k}-1)^2(\widetilde{k}^2-2\widetilde{k}+2t-2)+t^2-2t+2.$
        }}
    \end{equation}
% In that case, if $\widetilde{\D}$ is a $2$-$(\widetilde{v}, \widetilde{k}, \widetilde{\lambda})$ design, 
Moreover, if \eqref{hyp:con-rect} holds and $\widetilde{\D}$ is a $2$-$(\widetilde{v},\widetilde{k},\widetilde{\lambda})$ design, then $\D(\widetilde{\D},(e_1,e_2))$ is a $2$-$(v,k,\lambda)$ design with parameters 
    \begin{align*}
    v &= \widetilde{v}\cdot e_1\cdot e_2\\
    k &= \widetilde{k} + e_1+e_2 - 2\\
    \lambda &= \widetilde{\lambda} \cdot \widetilde{v}^{e_1+e_2-2}
    \end{align*}
and  \eqref{e:kv} holds, that is, $t=k(k-1)/(v-1)$.
% Moreover, if \eqref{hyp:con} holds,
    % \begin{equation} %\label{e:kv} 
    % \frac{k(k-1)}{v-1} = t.
    % \end{equation}
\end{proposition}

% However the expression for $e_1$ and $e_2$ in \eqref{hyp:con-rect} could conceivably yield $e_1=1$ so $e_2\geq 1$ (possibly equal). Note that $e_1\leqslant e_2$, so if $e_2=1$ then $e_1=1$ too. Assume $e_1=1$. As $e_1$ has the same value as $e$ in Proposition \ref{prop:con-wr-comb}, it follows that  $(\widetilde{v}, \widetilde{k},t)=(3,2,1)$, corresponding to the complete $2$-$(3,2,1)$ design $\widetilde{\D}$, in which case we compute that $e_2=1$ too. On the other hand the parameters $(\widetilde{v}, \widetilde{k},t,e_1,e_2)=(2,2,2,2,4)$ do correspond to an example of a $2$-design in Proposition~\ref{prop:con-rect-comb}, namely $\widetilde{\D}$ is the trivial $2$-$(2,2,1)$ design and $\D(\widetilde{\D},(2,4))$ is a $2$-$(16,6,16)$ design (which is not trivial nor complete). This shows that Construction~\ref{con:rect-comb} may yield interesting example even if the input design $\widetilde{\D}$ is trivial.}

\begin{proof}
For all pairs of distinct points $(p,(i,j))$ and $(q,(i',j'))$ in $\PP$, we determine the number of blocks of $\B$ containing both points. Here $p,q\in\widetilde{\PP}$, $i,i'\in\Delta_1$ and $j,j'\in\Delta_2$.

\medskip\noindent
\emph{Case $(i',j')=(i,j)$}:\quad 
In this case $p, q$ must be distinct.
Let $\widetilde{\lambda}(p,q)$ be the number of blocks in $\widetilde{\B}$ containing $p$ and $q$.
By Construction~\ref{con:rect-comb}, the only blocks that can contain these two points are in $\B_{(i,j)}$, and the number of such blocks is 
    \begin{equation} \label{eq:sameclass-rect}
    \widetilde{\lambda}(p,q)\cdot \rho_{e_1,e_2} \cdot \widetilde{v}^{e_1+e_2-2}.  
    \end{equation}
In particular, the number of blocks containing a point-pair of this type is independent of $p,q$ if and only if $\widetilde{\D}$ is a $2$-$(\widetilde{v}, \widetilde{k}, \widetilde{\lambda})$ design.

\medskip\noindent
\emph{Case $j'=j$ but $i'\neq i$}:\quad By Construction~\ref{con:rect-comb}, a block containing these two points must lie in $\B_{(\ell,j)}$ for some $\ell\in\Delta_1.$
The number of blocks in $\B_{(i,j)}$ containing these two points is $\widetilde{r} \cdot \widetilde{v}^{e_1-2}\cdot \rho_{e_1,e_2}\cdot \widetilde{v}^{e_2-1}$, which by \eqref{eq:vr=bk} is equal to $\widetilde{b}\cdot \widetilde{k} \cdot \rho_{e_1,e_2}\cdot\widetilde{v}^{e_1+e_2-4}$.
This is also the number of blocks in $\B_{(i',j)}$ containing the two points. Also if $e_1>2$ and $\ell \notin \{i,i'\}$, then the number of blocks in $\B_{(\ell,j)}$ containing the two points is $\widetilde{b} \cdot \widetilde{v}^{e_1-3}\cdot \rho_{e_1,e_2}\cdot \widetilde{v}^{e_2-1}$. Thus, the total number of blocks in $\B$ containing $(p,(i,j))$ and $(q,(i',j))$ is 
    \begin{equation} \label{eq:samecolumn-rect}
    2 \widetilde{b}\cdot \widetilde{k} \cdot \rho_{e_1,e_2}\cdot\widetilde{v}^{e_1+e_2-4}+ (e_1-2) \cdot \widetilde{b}\cdot \rho_{e_1,e_2}\cdot \widetilde{v}^{e_1+e_2-4}=\widetilde{b} \cdot \widetilde{v}^{e_1+e_2-4}\cdot \rho_{e_1,e_2}\cdot(2\widetilde{k}+e_1-2).
    \end{equation}
Note:  the numbers of blocks containing a given point-pair lying in either of these two cases is equal to some constant, independent of $p,q$, if and only if 
  \begin{equation} \label{eq:2design-rect}
    \text{$\widetilde{\D}$ is a $2$-$(\widetilde{v}, \widetilde{k}, \widetilde{\lambda})$ design  and }  \widetilde{b}\cdot (2\widetilde{k} + e_1 - 2) = \widetilde{\lambda} \cdot \widetilde{v}^2.
    \end{equation}
Note further that the condition in  \eqref{eq:2design-rect}  is the same as that in  \eqref{eq:2design-comb} with $e=e_1$. Thus the argument following \eqref{eq:2design-comb} is valid also in our case and yields that 
\begin{equation}\label{eq:2design-rect2}
\text{$t := \frac{\widetilde{k}(\widetilde{k}-1)}{{\widetilde{v}}-1}$ is an integer,  and $e_1 = ({\widetilde{k}}-1)({\widetilde{k}}-2) + t$.}    
\end{equation}
Thus, requiring a constant number of blocks containing  point-pairs of just these first two types already gives three of the four conditions  in \eqref{hyp:con-rect}.

\medskip\noindent
\emph{Case $i'=i$ but $j'\neq j$}:\quad By Construction~\ref{con:rect-comb},  a block containing these two points cannot lie in  $\B_{(i,\ell)}$ for any $\ell\in \Delta_2$.
For any $h\in\Delta_1\setminus\{i\}$, the number of blocks in $\B_{(h,j)}$ containing these two points is $\widetilde{b}\cdot \widetilde{v}^{e_1-2}\cdot \frac{\rho_{e_1,e_2}}{e_1-1}\cdot \widetilde{v}^{e_2-2}$. (Here we used \eqref{eq:nto1fctfixingone} since the image of the point $j'$ under the $m$-to-$1$ function $\Phi$ must equal $i$.) This is also the number of blocks in $\B_{(h,j')}$ containing the two points.

If $e_2=2$ then, since $e_1-1$ divides $e_2-1$, also $e_1=2$ and so $m=1$. In this case the discussion in the previous paragraph has covered all possibilities for blocks containing $(p,(i,j))$ and $(q,(i,j'))$, and  the number of blocks is  $2\cdot \widetilde{b}\cdot \rho_{e_1,e_2}\cdot \widetilde{v}^{e_1+e_2-4}$. We will show that this quantity matches the expression for larger $e_2$-values in \eqref{eq:samerow-rect}.

Assume now that $e_2>2$. For $h\neq i$ and  $\ell \notin \{j,j'\}$, the number of blocks in $\B_{(h,\ell)}$ containing the point-pair  is $\widetilde{b} \cdot \widetilde{v}^{e_1-1}\cdot \frac{(m-1)\rho_{e_1,e_2}}{(e_1-1)(e_2-2)}\cdot \widetilde{v}^{e_2-3}$, where we have used \eqref{eq:nto1fctfixingtwosameimage} since the $m$-to-$1$ function $(\Delta_2\setminus\{\ell\})\to (\Delta_1\setminus\{h\})$  defining the block in Construction~\ref{con:rect-comb} must map both $j$ and $j'$ to $i$.  
Thus the total number of blocks in $\B$ containing $(p,(i,j))$ and $(q,(i,j'))$, for $e_2>2$, is 
    \begin{align} \label{eq:samerow-rect}
   &2\cdot (e_1-1)\cdot \widetilde{b}\cdot \frac{\rho_{e_1,e_2}}{e_1-1}\cdot \widetilde{v}^{e_1+e_2-4}+(e_1-1)(e_2-2)\cdot \widetilde{b} \cdot \frac{(m-1)\rho_{e_1,e_2}}{(e_1-1)(e_2-2)}\cdot \widetilde{v}^{e_1+e_2-4} \\
   &= \widetilde{b}\cdot \widetilde{v}^{e_1+e_2-4}\cdot \rho_{e_1,e_2}\cdot (m+1) \notag
    \end{align}
and we note that this expression also holds for $e_2=2$, since $m=1$ in that case. 
Therefore, the  
number of blocks containing a point-pair of this third type is  equal to the constant number for point-pairs of the first two types if and only if the conditions of \eqref{eq:2design-rect}, \eqref{eq:2design-rect2}, and the following additional condition all hold: 
\begin{equation}\label{eq:rect-id}
m+1=2\widetilde{k} + e_1 - 2=2\widetilde{k}+({\widetilde{k}}-1)({\widetilde{k}}-2) + t-2={\widetilde{k}}^2-{\widetilde{k}}+ t.
\end{equation}
Since $m=(e_2-1)/(e_1-1)$, the second parameter  
$e_2=1+m(e_1-1)$ is determined by \eqref{eq:2design-rect2} and \eqref{eq:rect-id} as a function of $\widetilde{k}$ and $t$, namely
\begin{align*}
e_2 &= 1 + ({\widetilde{k}}^2-{\widetilde{k}} + t-1)(({\widetilde{k}}-1)({\widetilde{k}}-2) + t-1)\\
&= (\widetilde{k}-1)^2(\widetilde{k}^2-2\widetilde{k}+2t-2)+t^2-2t+2,
\end{align*}
which is the expression in \eqref{hyp:con-rect}.
 Finally we consider the point-pairs with $(i,j)$ and $(i',j')$ in general position.

\medskip\noindent
\emph{Case $i'\ne i$ and $j'\neq j$}:\quad 
The number of blocks in $\B_{(i,j)}$ containing these two points is $\widetilde{r}\cdot \widetilde{v}^{e_1-1}\cdot \frac{\rho_{e_1,e_2}}{e_1-1}\cdot \widetilde{v}^{e_2-2}$ (where we have used \eqref{eq:nto1fctfixingone} since the image of one point is already decided in the $m$-to-$1$ function). By \eqref{eq:vr=bk} this quantity is equal to $\widetilde{b}\cdot\widetilde{k}\cdot \frac{\rho_{e_1,e_2}}{e_1-1}\cdot \widetilde{v}^{e_1+e_2-4}$, and this is also the number of blocks in $\B_{(i',j')}$ containing the point-pair. 

By Construction~\ref{con:rect-comb},  a block containing these two points cannot lie in  $\B_{(i,\ell)}$ for any 
$\ell\neq j$ or  in $\B_{(i',\ell)}$ for any $\ell\neq j'$.
In particular, if $e_1=2$, then  we have exhausted all possibilities, and  so the number of blocks containing this point-pair is  $2\cdot \widetilde{b}\cdot\widetilde{k}\cdot \frac{\rho_{e_1,e_2}}{e_1-1}\cdot \widetilde{v}^{e_1+e_2-4}$. We will show that this quantity matches the expression in \eqref{eq:diffrowcolumn-rect} for larger $e_1$-values.

Now assume that $e_1>2$, and note that this implies that $e_2>2$ since $e_2\geq e_1$. 
% by the last comment of Remark \ref{rem:countingfunctions}).
For  $u \notin \{i,i'\}$, the number of blocks in $\B_{(u,j)}$ containing the point-pair is $\widetilde{b} \cdot \widetilde{v}^{e_1-2}\cdot \frac{\rho_{e_1,e_2}}{e_1-1}\cdot \widetilde{v}^{e_2-2}$, where again we have used \eqref{eq:nto1fctfixingone} since the image of one point is already decided in the $m$-to-$1$ function. 
This quantity is also the number of blocks in $\B_{(u,j')}$ containing the point-pair.
Finally for  $u \notin \{i,i'\}$ and  $\ell \notin \{j,j'\}$, the number of such blocks in $\B_{(u,\ell)}$ is
 $\widetilde{b} \cdot \widetilde{v}^{e_1-1}\cdot \frac{m \rho_{e_1,e_2}}{(e_1-1)(e_2-2)} \cdot \widetilde{v}^{e_2-3}$, and this time we have used \eqref{eq:nto1fctfixingtwodistinctimage} since the images of two points having distinct images are already decided in the $m$-to-$1$ function.

%In order for this number 
 
%
Thus the total number of blocks in $\B$ containing $(p,(i,j))$ and $(q,(i',j'))$ is
    \begin{align*}
    &2 \cdot \widetilde{b} \cdot \widetilde{k} \cdot \frac{\rho_{e_1,e_2}}{e_1-1} \cdot \widetilde{v}^{e_1+e_2-4}
    + 2(e_1-2) \cdot \widetilde{b} \cdot \frac{\rho_{e_1,e_2}}{e_1-1} \cdot \widetilde{v}^{e_1+e_2-4} \\
    &+(e_1-2)(e_2-2) \cdot \widetilde{b} \cdot \frac{m \rho_{e_1,e_2}}{(e_1-1)(e_2-2)} \cdot \widetilde{v}^{e_1+e_2-4} \\
    &= \widetilde{b}\cdot \frac{\rho_{e_1,e_2}}{e_1-1} \cdot \widetilde{v}^{e_1+e_2-4} \cdot  
\left(2 \widetilde{k} + 2(e_1-2) +m(e_1-2) \right)
 \end{align*}
 which we write as
 \begin{equation}\label{eq:diffrowcolumn-rect}
\widetilde{b}\cdot \frac{\rho_{e_1,e_2}}{e_1-1}\cdot \widetilde{v}^{e_1+e_2-4}\cdot  
\left(2 \widetilde{k}
+(e_1-2)(m+2) \right).
  \end{equation} 
 This expression agrees with the one we obtained above for the case $e_1=2$, and hence applies in all cases.

    In particular, the number of blocks containing a point-pair in this last case is equal to the number of blocks for point-pairs of the previous  type, if and only if the expressions in \eqref{eq:samerow-rect} and \eqref{eq:diffrowcolumn-rect} are equal, and this is true if and only if 
    \[
    m+1=\frac{2 \widetilde{k}
    +(e_1-2)(m+2)}{e_1-1}.
    \]
    This equality can be rewritten as $m+1=2 \widetilde{k}+e_1-2$, and holds if \eqref{eq:rect-id} does. Thus if the number of blocks containing a point-pair in the first three cases is constant, then this constant is equal to the number of blocks containing a point-pair in the last case (without any further restrictions). We therefore conclude that $\D(\widetilde{D},(e_1,e_2))$ is a 2-design if and only if all conditions in \eqref{hyp:con-rect} hold.

For the final assertions assume that  \eqref{hyp:con-rect} holds and that $\widetilde{\D}$ is a $2$-$(\widetilde{v},\widetilde{k},\widetilde{\lambda})$ design. Then the parameters of $\D(\widetilde{\D}, (e_1,e_2))$ are:   
$v=\widetilde{v}\cdot e_1\cdot e_2$, $\lambda = \widetilde{\lambda} \cdot \widetilde{v}^{e-1}$ (as computed in \eqref{eq:sameclass-rect}), and 
     $k = \widetilde{k} + e_1+e_2 - 2$. 
Finally, we note, by \eqref{hyp:con-rect}, that $e_1, e_2$ have the same values as in Remark \ref{rem:firsttwoes}, and that these are the first two values in Construction \ref{cons-iterated}. Thus it follows from Theorem \ref{thm:addchainontop} (with $n=2$) that \eqref{e:kv} holds.
\end{proof}

\begin{remark}
(a) We note that in Construction~\ref{con:rect-comb} the parameters satisfy $e_2\geq e_1\geq2$ and $\widetilde{k}\geq 2$, and in \eqref{hyp:con-rect} the integer $t$ is always positive. The expression in \eqref{hyp:con-rect} for $e_1$ then yields (since  $e_1\geq 2$) that either $\widetilde{k}\geq3$ or $t\geq2$. Moreover, when $\widetilde{k}\geq3$ then $e_1\geq 3$ and $e_2\geq 13.$

\medskip\noindent(b)  On the other hand if $\widetilde{k}=2$, then $e_1=t$ so we must have $t\geq 2$ and the expressions in \eqref{hyp:con-rect} yield $(t,\widetilde{v},e_1,e_2)=(2, 2, 2,4)$. These values correspond to a unique  $2$-design in Proposition~\ref{prop:con-rect-comb}, namely for $\widetilde{\D}$ the trivial $2$-$(2,2,1)$ design we obtain as $\D(\widetilde{\D},(2,4))$ a $2$-$(16,6,16)$ design which is neither trivial nor complete. One block for this design is shown in Figure \ref{fig:2-(16,6,16)}. This example demonstrates that Construction~\ref{con:rect-comb} may yield interesting designs even when applied with trivial input designs $\widetilde{\D}$.

\medskip\noindent
(c) Let $\D(\widetilde{\D},(2,4))$ as described in (b), with blocks as in Figure \ref{fig:2-(16,6,16)}. Then  $S_2 \wr (S_2\times S_4) \leqslant \Aut(\D)$ by Remark~\ref{remark-rectcombcons}. However, it can be checked with Magma \cite{magma} that the full automorphism group $\Aut(\D)$ is larger: namely  $\Aut(\D) = S_2\wr( S_4\wr S_2)$, preserving a 3-chain of partitions as described in \cite{chainspaper}. Indeed  the necessary and sufficient conditions to be a 2-design given in \cite{chainspaper} are satisfied:  since when $e_1=2$, the `columns partition' is irrelevant and the block-sets obtained as images of the block $B$ in  Figure \ref{fig:2-(16,6,16)}, under the two groups  $S_2 \wr (S_2\times S_4)$ and $S_2\wr( S_4\wr S_2)$, are the same.

\medskip\noindent
(d) It seems to us very unusual that the automorphism group of the design in part (c) is so much larger than the group  $\Aut(\widetilde{\D}) \wr (S_{e_1}\times S_{e_2})$ given by Remark~\ref{remark-rectcombcons}. For the reasons explained in (c), if $e_1=2$, then the  full automorphism group of the block design  $\D(\widetilde{\D}, (2,e_2))$  (not necessarily a $2$-design) certainly contains $\Aut(\widetilde{\D})\wr( S_{e_2}\wr S_{2})$ and so is larger than $\Aut(\widetilde{\D}) \wr (S_{e_1}\times S_{e_2})$. 
If $e_1\geq3$ then the  group $\Aut(\widetilde{\D})\wr( S_{e_2}\wr S_{e_1})$ does not preserve the block-set of $\D(\widetilde{\D}, (e_1,e_2))$, see Figure~\ref{fig:blockset-rect}.
We wonder if the example in (c) (and others with $e_1=2$) are exceptional, and we believe:
\begin{conjecture}
     If $e_1\geq 3$, then $\Aut(\D(\widetilde{\D}, (e_1,e_2)) )= \Aut(\widetilde{\D}) \wr (S_{e_1}\times S_{e_2})$. 
\end{conjecture}

\noindent
We checked this computationally on a few designs $\D(\widetilde{\D}, (e_1,e_2)$ for $e_1=3$, small values for $e_2$ and $\widetilde{\D}$ various small designs on $2$ or $3$ points (but no $2$-design as even the smallest example is too large to be able to check this computationally).
\end{remark}

\begin{figure}
    \centering
    \includegraphics{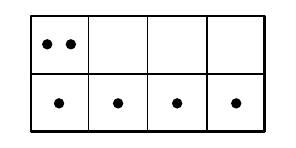}
    \caption{The design $\D(\widetilde{\D}, (2,4))$, where $\D$ is the trivial $2$-$(2,2,1)$ design}
    \label{fig:2-(16,6,16)}
\end{figure}

Next we consider the case where the input design $\widetilde{\D}=(\widetilde{\PP},\widetilde{\B})$ to Construction~\ref{con:rect-comb} admits a group $\widetilde{G}$ preserving a poset block structure on $\widetilde{\PP}$. We explain in Remark~\ref{rem:twoontop-poset} how the output design  $\D(\widetilde{\D},\{e_1,e_2\})$ also preserves a poset block structure for a larger poset.

\begin{remark}\label{rem:twoontop-poset}
For any subgroup $\widetilde{G} \leq \Aut(\widetilde{\D})$, the group $G = \widetilde{G} \wr (S_{e_1} \times S_{e_2})$ preserves the following partitions of $\widetilde{\PP} \times \Delta = \widetilde{\PP} \times (\Delta_1 \times \Delta_2)$:
    \begin{align*}
    \C_{\{1,2\}} &:= \{ \{\widetilde{\PP}\} \times \{(i,j)\} \ | \ (i,j) \in \Delta \}; \\ \C_{\{2\}} &:= \{ \{\widetilde{\PP}\} \times \Delta_1 \times \{j\} \ | \ j \in \Delta_2 \} \\
    \C_{\{1\}} &:= \{ \{\widetilde{\PP}\} \times \{i\} \times \Delta_2 \ | \ i \in \Delta_1 \}
    \end{align*}
    % \begin{itemize}
    % \item $\C_{\{1,2\}} := \{ \{\widetilde{\PP}\} \times \{(i,j)\} \ | \ (i,j) \in \Delta \}$; %which is represented by the set of squares in Figure \ref{fig:poset-con:rect};
    % \item $ \C_{\{2\}} := \{ \{\widetilde{\PP}\} \times \Delta_1 \times \{j\} \ | \ j \in \Delta_2 \}$; %which is represented by the set of large columns in Figure \ref{fig:poset-con:rect}; and
    % \item $ \C_{\{1\}} := \{ \{\widetilde{\PP}\} \times \{i\} \times \Delta_2 \ | \ i \in \Delta_1 \}$. %which is represented by the set of large rows in Figure \ref{fig:poset-con:rect}.
    % \end{itemize}
Each of these partitions is visualised by the diagram on the right hand side of Figure \ref{fig:poset-con:rect}. By  \cite[Theorem 2.5]{SmallEx}, the group $G$  acting on  $\C_{\{1,2\}}$ is permutationally isomorphic to the group $S_{e_1} \times S_{e_2}$ acting on $\Delta$, and the group induced on each $\C_{\{1,2\}}$-class by its setwise stabiliser is permutationally isomorphic to the $\widetilde{G}$-action on $\widetilde{\PP}$.
    \begin{figure}
        \includegraphics[]{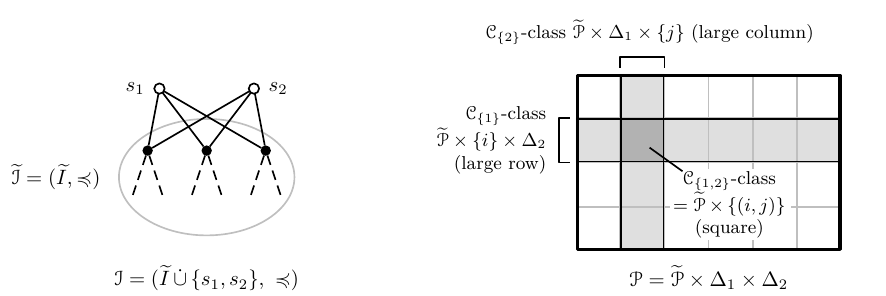}
        \caption{Poset after applying Construction \ref{con:rect-comb}}
        \label{fig:poset-con:rect}
    \end{figure}

Suppose that the group $\widetilde{G}$ in Construction \ref{con:rect-comb} preserves an $\widetilde{\I}$-poset block structure on $\widetilde{\PP}$  corresponding to a poset $\widetilde{\I} = (\widetilde{I},\preccurlyeq)$ as described in Subsection \ref{ss:posetblockstruc}. That is, suppose that $\widetilde{\PP} = \Delta_{\widetilde{I}} = \prod_{i \,\in\, \widetilde{I}} \Delta_i$ and that $\widetilde{G}$ is a subgroup of the generalised wreath product $\prod_{(\widetilde{I},\preccurlyeq)} (\Sym(\Delta_i),\Delta_i)$. Then $\widetilde{G}$ preserves the partitions $\widetilde{\C}_{\widetilde{J}}$ for all ancestral subsets $\widetilde{J} \subseteq \widetilde{I}$. If we visualise the set $\Delta=\Delta_1 \times \Delta_2$  as a rectangular array with $e_1$ rows and $e_2$ columns, the group $S_{e_1} \times S_{e_2}$ preserves three proper partitions of $\Delta$, namely, the partition into singletons, the set of rows, and the set of columns. Each partition $\widetilde{\C}_{\widetilde{J}}$ of $\widetilde{\PP}$ then gives rise to a partition of the set $\PP$; similar to the notation introduced in Remark \ref{rem:poset-con:wr-comb}, we denote this partition by $\widetilde{\C}_{\widetilde{J}} \times \Delta$, where
    \[
    \widetilde{\C}_{\widetilde{J}} \times \Delta = \big\{ \widetilde{C} \times \{(i,j)\} \ \big| \ \widetilde{C} \in \widetilde{\C}_{\widetilde{J}}, \ (i,j) \in \Delta \big\}.
    \]
In particular, the partition $\C_{\{1,2\}}$ above is $\widetilde{\C}_{\varnothing} \times \Delta$. For each ancestral subset $\widetilde{J}$, the group $G$ preserves the partition $\widetilde{\C}_{\widetilde{J}} \times \Delta$, and each    ($\widetilde{\C}_{\widetilde{J}} \times \Delta$)-class is contained in a unique ($\C_{\{1,2\}}$)-class  (represented by a square in the rectangle on the right hand side of Figure \ref{fig:poset-con:rect}). Thus inside each $\C_{\{1,2\}}$-class $\widetilde{\PP} \times \{(i,j)\}$ is a copy of the  poset block structure $(\widetilde{\PP}, \widetilde{\C}^*)$ corresponding to the poset $\widetilde{\I}$ and preserved by the group $\widetilde{G}$. It follows that $G$ preserves a poset block structure corresponding to the poset $\I = (I,\preccurlyeq)$, where: 
\begin{equation} \label{eq:twonodesontop}
I=\widetilde{I} \,\dot\cup\, \{s_1,s_2\},\ \text{the restriction} \ 
    \I|_{\widetilde{I}} = \widetilde{\I},\ 
    i\preccurlyeq s_1 \ \text{and} \ i \preccurlyeq s_2 \ \text{for all} \ i\in \widetilde{I}, s_1 \not\preccurlyeq s_2, s_2 \not\preccurlyeq s_1.      \end{equation}
That is, $\I$ is the poset obtained by adding two independent nodes `on top of' $\widetilde{\I}$, as illustrated in Figure \ref{fig:poset-con:rect}. 
The ancestral subsets in $\I$ are the sets $\varnothing$, $\{s_1\}$, $\{s_2\}$, and $\widetilde{J} \cup \{s_1,s_2\}$, where $\widetilde{J}$ is ancestral in $\widetilde{\I}$.  For each ancestral subset $\widetilde{J}$ in $\widetilde{\I}$, the partition $\widetilde{\C}_{\widetilde{J}} \times \Delta$ is the partition $\C_{\widetilde{J} \,\cup\, \{s_1,s_2\}}$.
\end{remark}

If we use the same parameters for the input design of Construction~\ref{con:rect-comb} as those used in Lemma~\ref{lem:chainontop-p+1} for the input design of Construction~\ref{con:wr-comb}, we see in Remark~\ref{rem:2nodesontop-p+1} that the output design has different, but somewhat similar parameters. 

\begin{remark} \label{rem:2nodesontop-p+1}
Under the hypotheses of Proposition \ref{prop:con-rect-comb}, suppose that the input $1$-design $\widetilde{\D}$ has parameters $\widetilde{k} = p+1$ and $\widetilde{v} = p^2 + p + 1$, for some positive integer $p$. Then $t = 1$, and the parameters of the design $\D$ arising from Construction \ref{con:rect-comb} are 
    \begin{equation} \label{eq:2nodesontop-p+1}
    e_1 = p^2 - p + 1, \ \ 
    e_2 = p^4 - p^2 + 1, \ \ 
    k = p^4 + 1, \ \ 
    v = p^8 + p^4 + 1
    \end{equation}
\end{remark}

In the next two examples we first show how to use Construction~\ref{con:rect-comb} to produce $2$-designs admitting a block-transitive group preserving a poset block structure for the V-shaped poset with three nodes. Then we give in Example~\ref{ex:Y} a recursive process to construct designs corresponding to arbitrarily large Y-shaped posets.  

\begin{example}[Designs corresponding to the V-shaped poset with three nodes] \label{ex:V}
Consider the V-shaped poset $\I = (I,\preccurlyeq)$ with $I = \{1,2,3\}$, such that $1 \prec 3$ and $1 \prec 2$, as shown in Figure \ref{fig:ex-V}.
    \begin{figure}
    \includegraphics{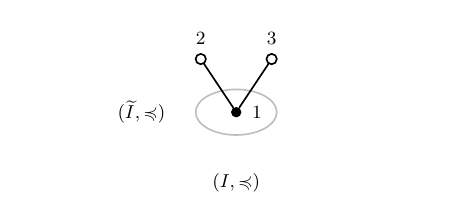}
    \caption{Poset for Example \ref{ex:V}}
    \label{fig:ex-V}
    \end{figure}
Then $I = \widetilde{I} \,\dot\cup\, \{2,3\}$ where $\widetilde{I} = \{1\}$, $\widetilde{I}$ and the pair $\{2,3\}$ satisfy condition \eqref{eq:twonodesontop}, and $(\widetilde{I},\preccurlyeq)$ is a single node.
A $2$-design admitting a block-transitive group preserving a V-shaped poset block structure on points can therefore be obtained from Construction \ref{con:rect-comb} with input any $2$-design $\widetilde{\D}=(\widetilde{\PP},\widetilde{\B})$ such that the parameters $\widetilde{v}$ and $\widetilde{k}$  satisfy condition \eqref{hyp:con-rect}. A family of such designs was given in \cite[Example 5.2]{SmallEx}, and could have been obtained from Construction \ref{con:rect-comb} by taking as $\widetilde{\PP}$ any set of size $p^2+p+1$, and as block-set $\widetilde{\B}$ the set of all images of any $(p+1)$-subset $B\subset \widetilde{\PP}$ under the elements of any $2$-transitive subgroup $\widetilde{G}$ of $\Sym(\widetilde{\PP})$.  (Note that this $\widetilde{\D}$ is the  $2$-design described in Example \ref{ex:chains}.) Then the output $2$-design $\D(\widetilde{\D}, (e_2,e_3))$ on applying Construction~\ref{con:rect-comb} has, by Remark \ref{rem:2nodesontop-p+1}, parameters  $e_2 = p^2 - p + 1$ and $e_3 = p^4 - p^2 + 1$.  (Note our change in notation: we have used $e_i$ for the size of the set $\Delta_i$ corresponding to node $i$ of the poset $\I$, see Figure~\ref{fig:ex-V}.)
\end{example}

\begin{example}[Designs corresponding to a Y-shaped poset] \label{ex:Y}
Consider a Y-shaped poset $\I = (I,\preccurlyeq)$ with $I = \{1,2, \ldots, s\}$ (for some $s \geq 4$), where $\widetilde{I} := \{1, 2, \ldots, s-2\}$ forms a chain with $1 \prec 2 \prec \ldots \prec s-2$, and $i \prec s-1$,\  $i \prec s$ for all $i \in \widetilde{I}$, $s-1 \not\preccurlyeq s$ and $s \not\preccurlyeq s-1$, as shown in Figure \ref{fig:ex-Y}. 
    \begin{figure}
        \centering
        \includegraphics{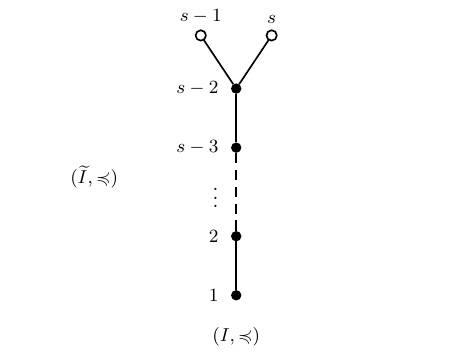}
        \caption{Poset for Example \ref{ex:Y}}
        \label{fig:ex-Y}
    \end{figure}
Then $I = \widetilde{I} \,\dot\cup\, \{s-1, \, s\}$ satisfies the conditions \eqref{eq:twonodesontop}. We can construct (recursively) a $2$-design $\D$ admitting a block-transitive group of automorphisms preserving an $\I$-poset block structure on points as follows: Take $\widetilde{\D}$ to be the design $\D^{s-2}$ in \cite[Construction 4.4]{chainspaper}, which has parameters $\widetilde{v} = p^2 + p + 1$ and $\widetilde{k} = p + 1$ where $p = q^{2^{s-3}}$ for some positive integer $q$. Let $e_{s-1} = p^2 - p + 1$ and $e_s = p^4 - p^2 + 1$ (so that $e_{s-1}$ and $e_s$ are the numbers $e_1$ and $e_2$, respectively, in Remark \ref{rem:2nodesontop-p+1}). Then the design $\D=\D(\widetilde{\D},(e_{s-1},e_s))$ is $\I$-imprimitive and is a $2$-design.  
\end{example}

Finally we observe in Example~\ref{ex:smallposets-2nodesontop} how $2$-designs relative to several other small posets can be constructed using Construction~\ref{con:rect-comb}. 

\begin{example}[Designs corresponding to other posets] \label{ex:smallposets-2nodesontop}
The $2$-designs constructed in \cite[Examples 5.1, 5.2, and 5.4]{SmallEx} and \cite[Construction 7.7]{multigrids} all satisfy condition \eqref{hyp:con-rect} with $t = 1$. Take $\widetilde{\D}$ to be one of these designs and let $\widetilde{\I}$ be the corresponding poset. Then applying Construction \ref{con:rect-comb}  produces a $2$-design admitting a block-transitive group preserving an  $\I$-poset block structure, where $\I$ is one of the posets pictured in Figure \ref{fig:smallposets-2nodesontop}, with the nodes of the poset $\widetilde{\I}$ indicated by the black dots.
\end{example}

\begin{figure}
    \centering
    \includegraphics[width=0.8\linewidth]{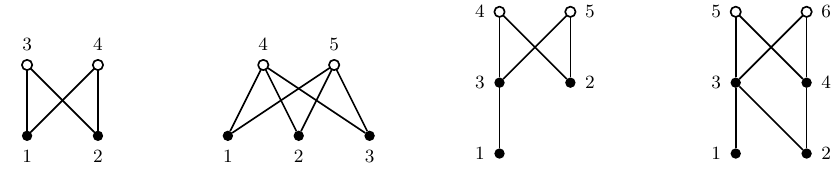}
    \caption{Posets for Example \ref{ex:smallposets-2nodesontop}}
    \label{fig:smallposets-2nodesontop}
\end{figure}

\section{Other posets}\label{s:other}

In this final section we give a collection of new design constructions to illustrate how Constructions~\ref{con:wr-comb} and~\ref{con:rect-comb} can be applied recursively in any order to obtain block-transitive poset-imprimitive $2$-designs for a variety of posets, including those in  Figure~\ref{fig:towers}. This flexibility is available to us because,  in both Constructions \ref{con:wr-comb} and \ref{con:rect-comb}, the parameter $t$  remains constant: that is to say, the integer value of $t$ for the input design is equal to the value of $t$ for the output design. Hence, if the input design $\widetilde{\D}$ for either of these constructions is a $2$-design, then applying any sequence of the two constructions produces another $2$-design with the same $t$ as $\widetilde{\D}$, provided that, at each step of the sequence, if Construction \ref{con:wr-comb} is applied then the parameter $e$ is chosen to satisfy \eqref{hyp:con}, or if Construction \ref{con:rect-comb} is applied then the parameters $e_1$ and $e_2$ are chosen to satisfy \eqref{hyp:con-rect}.

If the starting $2$-design $\widetilde{\D}$ has a block-transitive group preserving a poset block structure corresponding to the poset $\widetilde{\I} = (\widetilde{I},\preccurlyeq)$, then the $2$-design $\D$ obtained at the end of the sequence will admit a block-transitive group preserving a poset bock structure corresponding to a poset $\I = (\widetilde{I} \,\dot\cup\, H, \preccurlyeq)$, for some poset $\mathscr{H} = (H,\preccurlyeq)$ with the following properties:
    \begin{itemize}
        \item $\mathscr{H}$ has exactly one, or exactly two, minimal elements;
        \item $\mathscr{H}$ is built as follows: starting with an initial poset $\widetilde{\mathscr{H}}$ consisting of one node or two independent nodes, we apply a sequence of operations, where each operation either adds one node, or adds two independent nodes, on top of $\widetilde{\mathscr{H}}$.
        \item the restrictions $\I|_{\widetilde{I}} = \widetilde{\I}$ and $\I|_{H} = \H$.
        \item $i\prec j$ for all $i\in \widetilde{I}$ and $j\in H$.
       % } {\color{blue} is it OK to say this here?}
    \end{itemize}
   
Examples of such posets $\mathscr{H}$ can be found in Figure~\ref{fig:towers}. For instance, the poset $\mathscr{H}$ in Figure \ref{fig:towers}~(a) is obtained by starting with an initial poset consisting of  two independent nodes, and recursively adding two independent nodes on top of the current poset $\mathscr{H}$; the poset in Figure \ref{fig:towers} (b) is obtained by starting with an initial poset consisting of one node, and alternately adding two independent nodes, and adding one independent node, on top of the current poset. Procedures to obtain the posets in Figure~\ref{fig:towers}(c) and (d) can be described similarly.

 \begin{figure}
        \centering
        \includegraphics[width=0.7\linewidth]{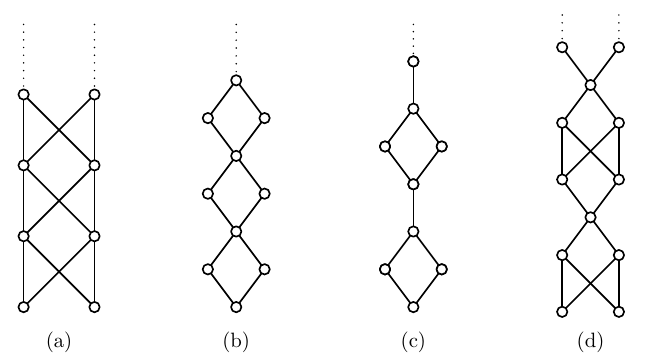}
        \caption{Some posets of arbitrary size that can be obtained by applying a sequence of Constructions \ref{con:wr-comb} and \ref{con:rect-comb}}
        \label{fig:towers}
    \end{figure}

\medskip

In Example~\ref{ex:4nodes}, we summarise what is known about the existence of block-transitive poset-imprimitive $2$-designs for posets with exactly four nodes. For posets with less than four nodes the existence question is completely resolved by  \cite[Theorem 1.3]{SmallEx}.  

\begin{example}[Designs corresponding to posets with four nodes]\label{ex:4nodes}
Figure \ref{fig:4nodes} shows all of the posets $\I$ with four nodes, which can be constructed by recursively adding one node, or two independent nodes, on top of an existing poset (as described above). %{\color{lightgray}apart from the `N-poset'} {\color{lightgray}\cite[Example 5.4]{SmallEx}, the disconnected union of two chains with two nodes each, and the V-shaped poset with one arm of length $2$ and another arm of length $1$ \cite{simplecond} (not sure though if both of these last two examples will be included in that paper)} 
For each of these posets $\I$, an explicit infinite family of block-transitive, $\I$-imprimitive $2$-designs $\D$ can be obtained, as explained above, from an initial $2$-design $\widetilde{\D}$ by applying a sequence of Constructions \ref{con:wr-comb} and \ref{con:rect-comb}. The $2$-design $\widetilde{\D}$ corresponds to the poset $\widetilde{\I}$ which is indicated by the black dots in Figure \ref{fig:4nodes}, and in each case there are possibilities for $\widetilde{\D}$ with the parameter $t$ being an integer, by \cite[Theorem 1.3]{SmallEx} (usually with $t=1$).  
%
% In fact, in each of these constructions we have $t=1$. For each of the posets $\I$ in Figure \ref{fig:4nodes},  the designs  $\D$ can be obtained, as explained above, from an initial design $\widetilde{\D}$ by applying a sequence of Constructions \ref{con:wr-comb} and \ref{con:rect-comb}. The design $\widetilde{\D}$ corresponds to the poset $\widetilde{\I}$ which is indicated by the black dots in Figure \ref{fig:4nodes}.
%, and we can take it to be either a block-transitive, point-primitive $2$-design (in the case of Figure \ref{fig:4nodes} (a), (c), and (f)), or a block-transitive, point-imprimitive $2$-design in the family of designs described in \cite[Construction 7.7]{multigrids} (for Figure \ref{fig:4nodes} (b)), \cite[Example 5.1]{SmallEx} (for Figure \ref{fig:4nodes} (g)), or one of the grid-imprimitive designs listed in Example \ref{ex:V-inv} (for Figure \ref{fig:4nodes} (d) and (e)). We note that an infinite family of examples corresponding to Figure \ref{fig:4nodes} (a) can also be obtained as a special case of {\color{lightgray}\cite[Construction 4.4]{chainspaper} by taking $s = 4$} {\color{magenta}Example \ref{ex:chains} by taking $n = 3$}. {\color{blue} isn't that construction exactly the same as what we are doing here? If so I think we can delete that sentence, or make that clearer.}
% {\color{blue} are there infinite families of examples for all these? In both cases the $t$ of $\D_0$ must be an integer.} {\color{magenta}Yes - for each example we use the usual $e_1, \ldots, e_4$, where the integer $p$ is arbitrary.}
%
In fact we may obtain appropriate choices for $\widetilde{\D}$ (with $t$ an integer) as follows.
\begin{itemize}
\item For Figure \ref{fig:4nodes} (a), (c), and (f): any block-transitive, point-primitive $2$-design, for example those mentioned in Remark \ref{r:con-wr-ex}.  These are the same designs used as input designs in Example \ref{ex:chains}, and indeed Example \ref{ex:chains} for $n=3$ gives the same construction as described here.
\item For Figure \ref{fig:4nodes} (b): any block-transitive, point-imprimitive $2$-design in the family of designs described in \cite[Construction 7.7]{multigrids}.
\item For Figure \ref{fig:4nodes} (g): any block-transitive, point-imprimitive $2$-design in the family of designs described in \cite[Example 5.1]{SmallEx}.
\item For Figure \ref{fig:4nodes} (d) and (e): any block-transitive, grid-imprimitive $2$-design. These are the same $2$-designs described as input designs in  Example \ref{ex:V-inv}.
\end{itemize}

In Figure \ref{fig:4nodes-2}, we list four posets
 with four nodes for which there are known  infinite families of examples, but which cannot be obtained by Constructions \ref{con:wr-comb} or \ref{con:rect-comb}. %{\color{lightgray}These are the N-poset, the V-poset where one arm has length $2$ and the other has length $1$, the disconnected union of two chains with two nodes each, and the disconnected union of a chain with three nodes and one isolated node.} {\color{blue} I think the drawings are clear enough that we don't need this sentence.} 
 A family of examples for the N-poset (Figure \ref{fig:4nodes-2}(h)) can be found in \cite[Example 5.4]{SmallEx}. J.M. Dacaymat has constructed examples for the other three posets as part of his Ph.D. thesis; examples for the two disconnected chains (Figure \ref{fig:4nodes-2} (j)) are included in a manuscript that is currently under preparation \cite{simplecond}.

There are five remaining posets with four nodes, for which, as far as we know, there are no known infinite families of examples of $2$-designs. These are shown in Figure \ref{fig:4nodes-3}. There is, however, a unique block-transitive $2$-design known for the 4-antichain (Figure \ref{fig:4nodes-3} (p)), see \cite[Theorem 1.4 and Example 7.10]{multigrids}. We thus pose the following problem:

\bigskip\noindent
\textbf{Problem.} Find infinite families of examples of block-transitive, point-imprimitive $2$-designs which admit a poset of invariant partitions that is isomorphic to one of the posets in Figure \ref{fig:4nodes-3}.

\end{example}
\begin{figure}
    \centering
    \includegraphics{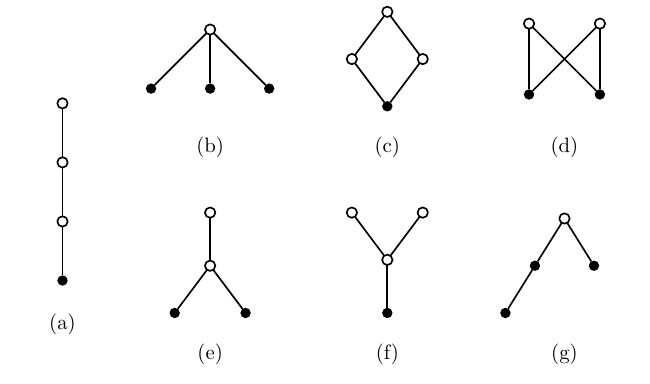}
    \caption{Posets with four nodes corresponding to known infinite families of examples of block-transitive, point-imprimitive $2$-designs that can be obtained using a sequence of Constructions \ref{con:wr-comb} and \ref{con:rect-comb} }
    \label{fig:4nodes}
\end{figure}

\begin{figure}
    \centering
    \includegraphics{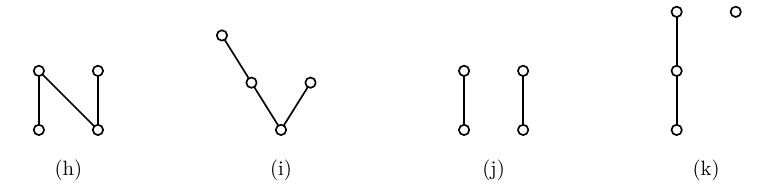}
    \caption{Posets with four nodes corresponding to known infinite families of examples of block-transitive, point-imprimitive $2$-designs that cannot be obtained using a sequence of Constructions \ref{con:wr-comb} and \ref{con:rect-comb}}
    \label{fig:4nodes-2}
\end{figure}

\begin{figure}
    \centering
    \includegraphics{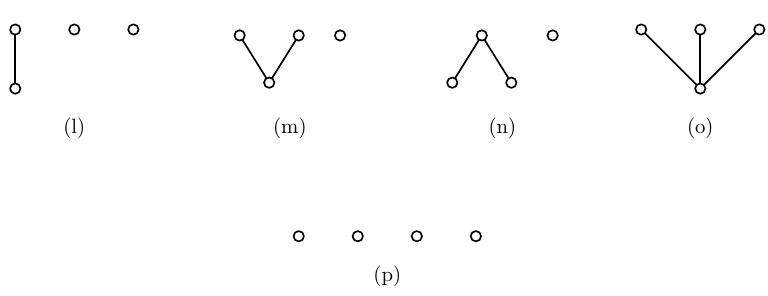}
    \caption{Posets with four nodes with no known corresponding infinite families of examples of block-transitive, point-imprimitive $2$-designs}
    \label{fig:4nodes-3}
\end{figure}

{}

\end{document}